\documentclass[11pt]{article}
\usepackage[margin=1in]{geometry}
\usepackage{amsmath,amsthm,amsfonts,amssymb}
\usepackage{graphicx,enumerate}
\graphicspath{ {./figures/} }
\usepackage[title]{appendix}

\usepackage[%dvips,
            CJKbookmarks=true, 
            bookmarksnumbered=true,
            bookmarksopen=true,
            colorlinks=true,
            citecolor=red,
            linkcolor=blue,
            anchorcolor=red,
            urlcolor=blue
            ]{hyperref}

\newcommand{\Exp}{\mathbb{E}}

\newcommand{\eps}{\varepsilon}
\renewcommand{\P}{\mathbb{P}}
\newcommand{\Q}{\mathbb{Q}}
\newcommand{\inner}[1]{\langle #1 \rangle}

\usepackage{tikz}
\usepackage{tikz-qtree}

\usetikzlibrary{calc}
\usetikzlibrary{decorations.pathreplacing,decorations.markings}

\tikzstyle{dot}=[circle,fill,black,inner sep=1pt]

\tikzset{
  on each segment/.style={
    decorate,
    decoration={
      show path construction,
      moveto code={},
      lineto code={
        \path [#1]
        (\tikzinputsegmentfirst) -- (\tikzinputsegmentlast);
      },
      curveto code={
        \path [#1] (\tikzinputsegmentfirst)
        .. controls
        (\tikzinputsegmentsupporta) and (\tikzinputsegmentsupportb)
        ..
        (\tikzinputsegmentlast);
      },
      closepath code={
        \path [#1]
        (\tikzinputsegmentfirst) -- (\tikzinputsegmentlast);
      },
    },
  },
  mid arrow/.style={postaction={decorate,decoration={
        markings,
        mark=at position .5 with {\arrow[#1]{stealth}}
      }}},
  early arrow/.style={postaction={decorate,decoration={
        markings,
        mark=at position .2 with {\arrow[#1]{stealth}}
      }}},
}

\def\alternatecolorred{%
    \pgfkeysalso{red}%
    \global\let\alternatecolor\alternatecolorblue % next time make it blue
}
\def\alternatecolorblue{%
    \pgfkeysalso{blue}%
    \global\let\alternatecolor\alternatecolorred % next time make it red
}
\newcommand{\altred}{\let\alternatecolor\alternatecolorred % first coordinate is red
\tikzset{every edge/.append code = {%
    \global\let\currenttarget\tikztotarget % save \tikztotarget in a global variable
    \pgfkeysalso{append after command={(\currenttarget)}}% automatically repeat it
			\alternatecolor
}}
}
\newcommand{\altblue}{\let\alternatecolor\alternatecolorblue % first coordinate is red
\tikzset{every edge/.append code = {%
    \global\let\currenttarget\tikztotarget % save \tikztotarget in a global variable
    \pgfkeysalso{append after command={(\currenttarget)}}% automatically repeat it
			\alternatecolor
}}
}

\tikzstyle{vertexdot}=[circle, draw, fill=black, minimum size=3,inner sep=0pt]

\newtheorem{theorem}{Theorem}
\newtheorem{lemma}{Lemma}
\newtheorem{proposition}{Proposition}

\newtheorem{definition}{Definition}

\newtheorem{remark}{Remark} 

\newtheorem{claim}{Claim}

\usepackage{color}
\usepackage{subfigure}
\usepackage{algorithm}% http://ctan.org/pkg/algorithms
\usepackage{algorithmic}% http://ctan.org/pkg/algorithms

\usepackage{xspace,prettyref}

\newcommand{\iiddistr}{{\stackrel{\text{\iid}}{\sim}}}

\newcommand{\reals}{{\mathbb{R}}}

\newcommand{\identity}{\mathbf I}

\newcommand{\Expect}{\mathbb{E}}
\newcommand{\expect}[1]{\mathbb{E}\left[ #1 \right]}

\newcommand{\expects}[2]{\mathbb{E}_{#2}\left[ #1 \right]}

\newcommand{\prob}[1]{ \mathbb{P}\left\{ #1 \right\} }

\newcommand{\var}{\mathsf{var}}

\def\independenT#1#2{\mathrel{\rlap{$#1#2$}\mkern2mu{#1#2}}}

\newcommand{\Hyp}{{\rm Hyp}}
\newcommand{\eg}{e.g.\xspace}
\newcommand{\ie}{i.e.\xspace}
\newcommand{\iid}{i.i.d.\xspace}
\newrefformat{eq}{(\ref{#1})}
\newrefformat{chap}{Chapter~\ref{#1}}
\newrefformat{sec}{Section~\ref{#1}}
\newrefformat{alg}{Algorithm~\ref{#1}}
\newrefformat{fig}{Fig.~\ref{#1}}
\newrefformat{tab}{Table~\ref{#1}}
\newrefformat{rmk}{Remark~\ref{#1}}
\newrefformat{clm}{Claim~\ref{#1}}
\newrefformat{def}{Definition~\ref{#1}}
\newrefformat{cor}{Corollary~\ref{#1}}
\newrefformat{lmm}{Lemma~\ref{#1}}
\newrefformat{prop}{Proposition~\ref{#1}}
\newrefformat{app}{Appendix~\ref{#1}}
\newrefformat{hyp}{Hypothesis~\ref{#1}}
\newrefformat{thm}{Theorem~\ref{#1}}
\newrefformat{ass}{Assumption~\ref{#1}}
\newrefformat{conj}{Conjecture~\ref{#1}}

\newcommand{\norm}[1]{\left\|{#1} \right\|_2}

\newcommand{\iprod}[2]{\left \langle #1, #2 \right\rangle}

\newcommand{\indc}[1]{{\mathbf{1}_{\left\{{#1}\right\}}}}

\newcommand{\diag}[1]{\mathsf{diag} \left\{ {#1} \right\} }

\newcommand{\calE}{{\mathcal{E}}}

\newcommand{\calN}{{\mathcal{N}}}

\newcommand{\calT}{{\mathcal{T}}}

\DeclareMathAlphabet{\varmathbb}{U}{bbold}{m}{n}

\renewcommand{\d}{{\rm d}}
\renewcommand{\hat}{\widehat}
\renewcommand{\tilde}{\widetilde}

\usepackage{bbm}

\newcommand{\MSE}{\text{MSE}}

\newcommand{\TV}{\mathrm{TV}}

\pgfdeclarelayer{background}
\pgfdeclarelayer{foreground}
\pgfsetlayers{background,main,foreground}

\usepackage{mathtools}
\DeclarePairedDelimiter\ceil{\lceil}{\rceil}
\DeclarePairedDelimiter\floor{\lfloor}{\rfloor}

\newcommand\independent{\protect\mathpalette{\protect\independenT}{\perp}}
\def\independenT#1#2{\mathrel{\rlap{$#1#2$}\mkern2mu{#1#2}}}

\theoremstyle{plain}

\newtheorem*{theorem*}{Theorem}
\newtheorem*{proposition*}{Proposition}

\author{
{ Galen Reeves}\thanks{Duke University; e-mail: {\tt galen.reeves@duke.edu}. }
\and
{ Jiaming Xu}\thanks{The Fuqua School of Business, Duke University; e-mail: {\tt jiamingxu.868@duke.edu}. }
\and
{ Ilias Zadik}\thanks{Operations Research Center, Massachusetts Institute of Technology ; e-mail: {\tt izadik@mit.edu}}
}

\begin{document}

\title{The All-or-Nothing Phenomenon in Sparse Linear Regression}

\maketitle

\begin{abstract}
We study the problem of recovering a hidden binary $k$-sparse $p$-dimensional vector $\beta$ from $n$ noisy linear observations $Y=X\beta+W$ where $X_{ij}$ are i.i.d.\  $\mathcal{N}(0,1)$ and $W_i$ are i.i.d.\  $\mathcal{N}(0,\sigma^2)$. A closely related  hypothesis testing problem is to distinguish the pair $(X,Y)$ generated from this structured model from a corresponding null model where $(X,Y)$ consist of purely independent Gaussian entries. In the low sparsity $k=o(p)$ and high signal to noise ratio $k/\sigma^2=\Omega\left(1\right)$ regime, we establish an ``All-or-Nothing'' information-theoretic phase transition at a critical sample size $n^*=2 k\log \left(p/k\right) /\log \left(1+k/\sigma^2\right)$, resolving a conjecture of \cite{gamarnikzadik}. Specifically, we show that if $\liminf_{p\to \infty} n/n^*>1$, then the maximum likelihood estimator almost perfectly recovers the hidden vector with high probability
     and moreover the true hypothesis can be detected with a vanishing error probability. Conversely,
     if $\limsup_{p\to \infty} n/n^*<1$, then it becomes information-theoretically impossible even to  recover an arbitrarily small but fixed fraction of the hidden vector support, or to test hypotheses strictly better than random guess.
     
     Our proof of the impossibility result builds upon two key techniques, which could be of independent interest. First, we use a conditional second moment method to upper bound the Kullback-Leibler (KL) divergence between the structured and the null model. Second, inspired by the celebrated area theorem, we establish a lower bound to the minimum mean squared estimation error of the hidden vector in terms of the KL divergence between the two models.
\end{abstract}

\tableofcontents

\section{Introduction}

In this paper, we study the information-theoretic limits of  the Gaussian sparse linear regression problem.  Specifically, for $n,p,k \in \mathbb{N}$ with $k \leq p$ and $\sigma^2>0$ we consider two independent matrices 
$X \in \reals^{n \times p}$ and $W \in \reals^{n \times 1}$
with $X_{ij} \iiddistr \calN(0,1)$
and $W_i \iiddistr \calN(0, \sigma^2)$, 
and observe
\begin{align}
Y = X \beta + W, \label{eq:planted}
\end{align}
where  $\beta$ is assumed to be uniformly chosen at random from
the set $\{ v \in \{0,1\}^p: \| v \|_0 = k\}$ and independent of $(X, W)$. 
The problem of interest is to recover $\beta$ given the knowledge of $X$ and $Y$. Our focus will be on identifying the minimal sample size $n$ for which the recovery is information-theoretic possible.

The problem of recovering the support of a hidden sparse vector $\beta \in \mathbb{R}^p$ given noisy linear observations has been extensively analyzed in the literature, as it naturally arises in many contexts including subset regression, e.g.~\cite{miller:1990}, signal denoising, e.g.~\cite{SignDen}, compressive sensing, e.g.~\cite{candes2005decoding}, \cite{donoho2006compressed}, information and coding theory, e.g.~\cite{joseph:2012}, as well as high dimensional statistics, e.g. \cite{wainwright:2009,wainwright2009sharp}. 
 The assumptions of Gaussianity of the entries of $(X,W)$ are standard in the literature. 
Furthermore, much of the literature (e.g.  \cite{aeron:2010},  \cite{ndaoud2018optimal}, \cite{wang2010information})  assumes a lower bound $\beta_{\min}>0$ for the smallest magnitude of a nonzero entry of $\beta$, that is $\min_{i: \beta_i \not = 0} |\beta_i| \geq \beta_{\min}$, as otherwise identification of the support of the hidden vector is in principle impossible. In this paper we adopt a simplifying assumption by focusing only on binary vectors $\beta$, similar to other papers in the literature such as \cite{aeron:2010}, \cite{gamarnikzadik} and \cite{gamarnikzadik2}. In this case recovering the support of the vectors is equivalent to identifying the vector itself. 

To judge the recovery performance we focus on the mean squared error (MSE). That is, given an estimator $\hat{\beta}$ as a function of $(X,Y)$, define mean squared error as 
$$
\MSE \left(\hat{\beta} \right)  \triangleq \expect{\| \hat{\beta} - \beta \|^2 },
$$
where $\|v \|$ denotes the $\ell_2$ norm of a vector $v$. In our setting, 
one can simply choose $\hat{\beta}=\expect{ \beta}$, which equals $\frac{k}{p} (1,1,\ldots,1)^\top$, and obtain a trivial $\MSE_0 =\expect{\|  \beta - \expect{ \beta}\|^2}$, which equals $k\left(1-\frac{k}{p}\right)$. We will adopt the following two natural notions of recovery, by comparing the $\MSE$ of an estimator $\hat{\beta}$ to $\MSE_0$. 

\begin{definition}[Strong and weak recovery]
We say that $\hat{\beta}=\hat{\beta}(Y,X) \in \mathbb{R}^p$ achieves
 \begin{itemize}

\item strong recovery if $\limsup_{p \to \infty} \MSE \left(\hat{\beta} \right)/\MSE_0=0$;
\item weak recovery if  $\limsup_{p \to \infty} \MSE \left(\hat{\beta} \right)/\MSE_0 <1.$

\end{itemize}
\end{definition}The fundamental question of interest in this paper is when $n$ as a function of $(p, k, \sigma^2)$ is such that
strong/weak recovery is information-theoretically possible.

The focus of this paper will be on sublinear sparsity levels, that is  on $k=o\left(p \right)$. A great amount of literature has been devoted on the study of the problem in the linear regime where $n,k,\sigma=\Theta(p).$
One line of work has provided upper and lower bounds on the accuracy of support recovery as a function of the problem parameters, \eg \cite{aeron:2010,reeves:2012,reeves:2013b,scarlett:2017}. Another line of work has derived explicit formulas for the minimum MSE (MMSE) $\expect{ \| \beta - \expect{ \beta \mid X, Y} \|^2}$. These formulas were first obtained heuristically using the replica method from statistical physics \cite{tanaka:2002, guo:2005} and later proven rigorously in \cite{reeves:2016a,barbier:2016}. However, to our best of knowledge, none of the rigorous techniques of  \cite{reeves:2016a,barbier:2016} apply when $k=o(p)$. 
Although there has been significant work focusing directly  on the sublinear sparsity regime, the identification of the exact information theoretic threshold of this fundamental statistical problem remains largely open (see \prettyref{sec:comparison} for a detailed discussion). 
 
Obtaining a tight characterization of the information-theoretic threshold is the main contribution of this work.

Towards identifying the information theoretic limits of recovering $\beta$, and out of independent interest, we also consider a closely related hypothesis testing problem, where the goal is to distinguish the pair $(X,Y)$ generated according to \eqref{eq:planted} from a model where both $X$ and $Y$ are independently generated. More specifically, given two independent matrices 
$X \in \reals^{n \times p}$ and $W \in \reals^{n \times 1}$
with $X_{ij} \iiddistr \calN(0,1)$
and $W_i \iiddistr \calN(0, \sigma^2)$,  we define
\begin{align}
Y \triangleq \lambda W, \label{eq:null_model}
\end{align}
where $\lambda > 0$ is a scaling parameter. We refer to the Gaussian linear regression model \prettyref{eq:planted} as the planted model, denoted by $P=P(X, Y)$, and \prettyref{eq:null_model} as the null model denoted by  $Q_{\lambda}=Q_{\lambda}(Y,X)$. We focus on characterizing the total variation distance $\mathrm{TV} \left(P,Q_{\lambda}\right)$ for various values of $\lambda$. One choice of particular interest is $\lambda=\sqrt{k/\sigma^2+1}$, under which $\expect{YY^\top}=(k+\sigma^2) \identity$ in both the planted and null models.

Analogous to recovery, we adopt the following two natural notions of testing~\cite{PerryWeinBandeira16,alaoui2017finite}. 

\begin{definition}[Strong and weak detection] Fix two probability measures $\mathbb{P},\mathbb{Q}$ on our observed data $(Y,X)$.
We say a test statistic $\calT(X, Y)$ with a threshold $\tau$ achieves 
\begin{itemize}
\item strong detection if
\[
\limsup_{p \to \infty} \left[ \P( \calT(X, Y) < \tau) + \Q(\calT(X, Y) \ge \tau) \right] = 0,
\]

\item weak detection, if 
$$\limsup_{p \to \infty} [\P( \calT(X, Y) < \tau ) + \Q( \calT(X, Y) \ge \tau )]<1.$$

\end{itemize}
\end{definition}
Note that strong detection asks for the test statistic to determine 
with high probability whether $(X,Y)$ is drawn from $\P$ or $\Q$, while
weak detection, similar to weak recovery, only asks for the test statistic to strictly outperform the
random guess. Recall that 
$$
\inf_{\calT,\tau} \left[ \P( \calT(X, Y) < \tau) + \Q(\calT(X, Y) \ge \tau) \right] = 1- \mathrm{TV}(P,Q).
$$
Thus equivalently, 
strong detection is possible if and only if 
$\liminf_{p\to \infty} \mathrm{TV}(\P,\Q) = 1$, 
and weak detection is possible if and only if 
$\liminf_{p \to \infty} \mathrm{TV}(\P,\Q)>0.$ The fundamental question of interest is when $n$ as a function of $(p, k, \sigma^2)$ 
is such that 
strong/weak detection is information-theoretically possible.

\subsection{Contributions}

Of fundamental importance is the following sample size:
\begin{align}
n^* \triangleq \frac{2k \log (p/k)}{\log (1+ k/\sigma^2)}. \label{eq:def_n_star}
\end{align}

We establish that $n^*$ is a sharp phase transition point for the recovery of $\beta$ when $k=o(\sqrt{p})$ and the signal to noise ratio
$k/\sigma^2$ is above a sufficiently large constant. In particular, for an arbitrarily small but fixed constant $\epsilon>0$, when $n < (1-\epsilon) n^*$, \textit{weak recovery} is impossible, but when $n>(1+\epsilon) n^*$, \textit{strong recovery} is possible. This implies that the 
rescaled MMSE undergoes a jump  from $1$ to $0$ at $n^*$ samples up to a small window
of size $\epsilon n$. 

We state this in the following Theorem,  which  summarizes the Theorems \ref{thm:main}, \ref{thm:recovery}, \ref{thm:positive} and \ref{thm:posDetection} from the main body of the paper. 

\begin{theorem*}[All-or-Nothing Phase Transition]
Let  $\delta \in (0,\frac{1}{2})$ and $\epsilon \in (0,1)$ be two arbitrary but fixed constants. 
Then there exists a constant $C(\delta,\epsilon)>0$ only depending only $\delta$ and $\epsilon$,
such that  if $k/\sigma^2 \ge C(\delta,\epsilon)$, then \begin{itemize}
\item When $k \le p^{\frac{1}{2}-\delta}$  and $$n <  \left(1-\epsilon \right) n^*, $$ both weak recovery of $\beta$ from $(Y,X) \sim P$ and weak detection between $P$ and $Q_{\lambda_0}$ are information-theoretically impossible, where $\lambda_0=\sqrt{\frac{k}{\sigma^2}+1}$.
\item  When  $k=o(p)$ and $$n>  \left(1+\epsilon \right) n^*, $$ both strong recovery of $\beta$ from $(Y,X) \sim P$ and $(\dagger)$ strong detection between $P$ and $Q_{\lambda}$  are information-theoretically possible for any $\lambda>0$. 

$(\dagger)$: strong detection requires an additional assumption $1+ k/\sigma^2 \leq  
\left( k \log \left(p/k\right) \right)^{1-\eta}$ for some arbitrarily small but fixed constant $\eta>0.$
\end{itemize}

\end{theorem*}

Note that the theorem above assumes $\sigma>0$. 
In the extreme case where $\sigma=0$, $n^*$ trivializes to zero and we can directly argue that one sample suffices for strong recovery. In fact, for any $\beta \in \{0,1\}^p$ and $Y_1=\inner{X_1,\beta}$ for $X_1 \sim \mathcal{N}(0,\identity_p)$, we can identify $\beta$ as the unique binary-valued solution of $Y_1=\inner{X_1,\beta}$, almost surely with respect to the randomness of $X$ (see e.g. \cite{gamarnikzadik3})

Note that the first part of the above result focuses on $k \le p^{1/2-\delta}$. It turns out that this is not a technical artifact and $k =o\left(p^{1/2}\right)$ is needed for $n^*$ to be the weak detection sample size threshold.
More details can be found in \prettyref{app:assumption_needed}. 
The sharp information-theoretic threshold for either detection or recovery is still open 
when $k=\Omega\left(p^{1/2}\right)$ and $k=o(p)$.

\paragraph{The phase transition role of $n^*$}

According to our main result, the rescaled minimum mean squared error of the problem, $\mathrm{MMSE}/\MSE_0$,  exhibits  a step behavior asymptotically. Loosely speaking, when $n<n^*$ it equals to one and when $n>n^*$ it equals to zero. We next intuitively explain why  such a step behavior for sparse high dimensional regression occurs at $n^*$, using ideas related to \textit{the area theorem}. The area theorem has been used in the channel coding literature
to study the MAP decoding threshold~\cite{measson2008maxwell} and the capacity-achieving codes~\cite{kudekar2017reed}. The approach described below is similar to the one used previously for linear regression \cite{reeves:2016a}.

First let us observe that $n^*$ is asymptotically equal to the \emph{ratio} of entropy  $H(\beta)=\log \binom{p}{k}$  and 
Gaussian channel capacity $\frac{1}{2} \log (1+k/\sigma^2)$. We explore this coincidence in the following way.
Let $I_n \triangleq I(Y_1^n; X, \beta)$ denote the mutual information between
 $\beta$ and  $(Y_1^n; X)$ with a total of $n$ linear measurements. 
Since the mutual information in the Gaussian channel under a second moment constraint is maximized by the Gaussian
input distribution, it follows that the increment of mutual information 
$I_{n+1}- I_n \le \frac{1}{2} \log (1 + \mathrm{MMSE}_n/\sigma^2)$, where $\mathrm{MMSE}_n$ denotes the minimum
MSE with $n$ measurements. In particular, all the increments are between zero and $\frac{1}{2} \log (1 +k/\sigma^2)$ and by telescopic summation for any $n$: 
\begin{align}\label{area} I_n \leq \frac{n}{2} \log (1 +k/\sigma^2), \end{align} 
with equality only if for all $m<n$,  $\mathrm{MMSE}_m=k$. This is illustrated in~\prettyref{fig:phase_diagram} where we plot $n$ against $I_{n+1}-I_n$.

Suppose now that we have established that strong recovery is achieved with $n^*=\frac{H(\beta)}{\frac{1}{2} \log (1 +k/\sigma^2)}$ samples.  

Then strong recovery and standard identities connecting mutual information and entropy implies that  $$I_{n^*}=H(\beta)=\frac{n^*}{2} \log (1 +k/\sigma^2).$$ In particular, (\ref{area}) holds with  equality, which means for all $n \le n^*-1$,
$\mathrm{MMSE}_n=k$. In particular, for all $n<n^*$, weak recovery is impossible. This area theorem 
is the key underpinning our converse proof of the weak recovery.

\begin{figure}[ht]
\centering
\includegraphics[width=0.5\columnwidth]{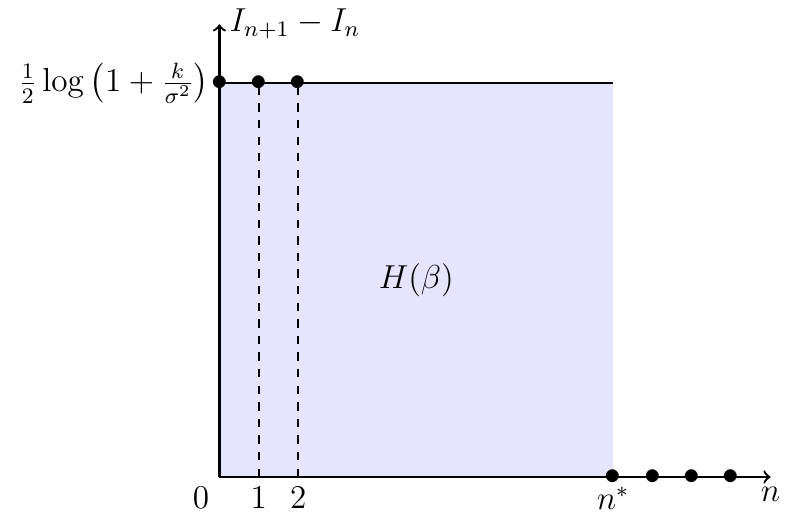}
\caption{The phase transition diagram in Gaussian sparse linear regression. 
The $y$-axis is the increment of mutual information with one additional measurement.
The area of blue region equals the entropy $H(\beta)\sim k \log (p/k)$.
}
\label{fig:phase_diagram}
\end{figure}

\subsection{Comparison with Related Work} \label{sec:comparison}

The information-theoretic limits of high-dimensional sparse linear regression 
have been studied extensively and there is a vast literature of multiple decades of research. In this section we focus solely on
the Gaussian and binary setting and furthermore on the results applying to high values of signal to noise ratio and sublinear sparsity.

\paragraph*{Information-theoretic Negative Results for weak/strong recovery  }

For the impossibility direction, previous work \cite[Theorem 5.2]{aeron:2010} has established that as $p \to \infty$, 
achieving $\mathrm{MSE}(\hat{\beta}) \leq d$ for any $d \in [0,k]$ is information-theoretically impossible if
$$
n \le 2 p \frac{h_2(k/p) - h_2(d/p )}{ \log \left( 1+ k/\sigma^2 \right)},
$$
where $h_2(\alpha)=-\alpha \log \alpha - (1-\alpha)\log (1-\alpha)$ for $\alpha \in[0,1]$
is the binary entropy function. This converse result is proved via a simple rate-distortion argument (see, e.g.~\cite{wu2018statistical} for an exposition). 
In particular, given any estimator $\hat{\beta}(X,Y)$ with $\MSE(\hat{\beta}) \le d$, we have
$$
p \left( h_2(k/p) - h_2(d/p ) \right) \le \inf_{\MSE(\tilde{\beta}) \le d } I(\tilde{\beta}; \beta) 
\le I(\hat{\beta}; \beta) \le I(X, Y; \beta)  \le \frac{n}{2} \log \left( 1+ k/\sigma^2 \right).
$$Notice that since $k=o(p)$ the result implies that if
$n \le \left(1-o\left(1\right) \right) n^*$, \emph{strong} recovery, that is $d=o(k)$, is information-theoretically impossible  and if $n=o(n^*)$, \emph{weak} recovery, that is $d\leq (1-\epsilon)k$ for an arbitrary $\epsilon \in (0,1)$,  is impossible.

More recent work \cite[Corollary 2]{scarlett:2017} further quantified the fraction of support that can be recovered when $n<(1-\epsilon)n^*$ for some fixed constant $\epsilon>0$. Specifically with $k=o(p)$ and any scaling of $k/\sigma^2$, if $n<(1-\epsilon)n^*$,   
then  the fraction of the support of $\beta$ that can be recovered correctly is at most $1-\epsilon$ with high probability; thus strong recovery is impossible.  

Restricting to the Maximum Likelihood Estimator (MLE) performance of the problem, it is shown in \cite{gamarnikzadik}  that under significantly small sparsity $k=O\left(\exp\left(\sqrt{\log p}\right)\right)$ and $k/\sigma^2 \rightarrow +\infty$, if $n \leq (1-\epsilon)n^*$,  the MLE not only fails to achieve strong recovery, but also fails to weakly recover the vector, that is recover correctly any positive constant fraction of the support.

 Our result (Theorem \ref{thm:recovery}) establishes that the MLE performance is fundamental. It improves upon the negative results  in the literature by identifying a sharp threshold for weak recovery, showing that if $k=o\left(\sqrt{p}\right)$, $k/\sigma^2 \geq C$ for some large constant $C>0$, and $n \le \left(1-\epsilon \right) n^*$, then \emph{weak} recovery is
information-theoretically impossible  by any estimator $\hat{\beta}(Y,X)$. 
In other words, no constant fraction of the support is recoverable under these assumptions.

\paragraph*{Information-theoretic Positive Results  for weak/strong recovery  }

In the positive direction, previous work \cite[Theorem 1.5]{akcakaya:2010} shows that when $k=o(p)$, $k/\sigma^2=\Theta(1)$, and $n>C_{k/\sigma^2} k \log (p-k)$ for  some $C_{k/\sigma^2}$, it is information theoretically possible to weakly recover the hidden vector. 

Albeit very similar to our results,  our positive result (Theorem \ref{thm:positive}) identifies the explicit value of $C_{k/\sigma^2}$ for which both weak and strong recovery are possible, that is $C_{k/\sigma^2}=2/\log \left(1+k/\sigma^2\right)$ for which $C_{k/\sigma^2} k \log (p/k)=n^*$.

In \cite{gamarnikzadik} it is shown that when $k=O\left(\exp \left( \sqrt{\log p}\right)\right)$ and $k/\sigma^2 \rightarrow + \infty$ then if $n \ge (1+\epsilon) n^*$ for some fixed $\epsilon>0$, \emph{strong} recovery is achieved by the MLE of the problem. We improve upon this result with \prettyref{thm:positive} by showing that when $n \ge (1+\epsilon) n^*$ for some fixed $\epsilon>0$ and any $k \leq cp$ for some $c>0$, then there exists a constant $C>0$ such that $k/\sigma^2 \geq C$ the MLE achieves  \emph{strong} recovery. In particular, we significantly relax the assumption from \cite{gamarnikzadik} by showing that MLE achieves \emph{strong} recovery with $(1+\epsilon)n^*$ samples for (1) any sparsity level less than $cp$ and (2) finite but large values of signal to noise ratio.

\paragraph*{Exact asymptotic characterization of MMSE for linear sparsity}

For both weak and strong recovery, the central object of interest is the MMSE $\expect{ \| \beta - \expect{ \beta \mid X, Y} \|^2}$ and its asymptotic behavior. While the asymptotic behavior of the MMSE remains a challenging open problem when $k=o(p)$, it has been accurately understood when $k=\Theta(p)$ and $k/\sigma^2 = \Theta(1)$.

To be more specific, consider the asymptotic regime where $k=\eps p$, $\sigma^2=k/\gamma$, and $n= \delta p$, for fixed positive constants $\eps, \gamma, \delta$ as $p \to +\infty$. 
The asymptotic minimum mean-square error (MMSE) can be characterized explicitly in terms of  $(\eps,\gamma, \delta)$. 

This characterization was first obtained heuristically using the replica method from statistical physics \cite{tanaka:2002, guo:2005} and later proven rigorously \cite{reeves:2016a,barbier:2016}. 
More specifically, for fixed $(\eps, \gamma)$,  let the asymptotic MMSE as a function of $\delta$ be defined by
\[
\mathcal M_{\eps,\gamma}(\delta)= \lim_{p \to \infty}  \frac{\expect{ \| \beta - \expect{ \beta \mid X, Y} \|^2}}{\expect{ \| \beta - \expect{ \beta } \|^2}} . 
 \]  
The results in \cite{reeves:2016a,barbier:2016} lead to an explicit formula for $\mathcal M_{\eps,\gamma}(\delta)$. Furthermore, they show that for  $\eps \in (0,1)$ and all  sufficiently large $\gamma \in (0, \infty)$,  $\mathcal M_{\eps,\gamma}(\delta)$ has a jump discontinuity as a function of $\delta$. 
The location of this discontinuity, denoted by $\delta^* = \delta^*(\eps, \gamma)$, occurs at  a value that is strictly greater than the threshold $n^*/p$. 

Furthermore,  at the the discontinuity,
the MMSE transitions  from a value that is strictly less than the MMSE without any observations  
to a value that is strictly positive, \ie, $\mathcal{M_{\eps,\gamma}}(0) > \lim_{ \delta \uparrow \delta^*} \mathcal M_{\eps,\gamma}(\delta)
>\lim_{ \delta \downarrow \delta^*} \mathcal M_{\eps,\gamma}(\delta) >0$.

To compare these formulas to the sub-linear  sparsity studied in this paper, one can consider the limiting behavior of 
$\mathcal M_{\eps,\gamma}(\delta)$ as $\eps$ decreases to zero. It can be verified that $\mathcal M_{\eps,\gamma}(\delta)$ converges indeed to a step zero-one function as $\eps \to 0$ and the jump discontinuity transfers indeed to the critical value $n^*/p$ which makes the behavior consistent with the results in this paper. 

However, an important difference is that the results in this paper are derived directly under the scaling regime $k = o(p)$ whereas the derivation described above requires one to first take the asymptotic limit $p \to \infty$ for fixed $(\epsilon, \gamma)$ and then take $\epsilon \to 0$. Since the limits cannot interchange in any obvious way, the results in this paper cannot be derived as a consequence of the rigorous results in \cite{reeves:2016a,barbier:2016}. Finally, it should be mentioned that taking the limit $\epsilon \to 0$ for the replica prediction suggests the step behavior for all values of signal-to-noise ratio  $\gamma$ (see Figure \ref{fig:replica}). In this paper, the step behavior is rigorously proven in the high signal-to-noise ratio regime. The proof of the step behavior when the signal-to-noise ratio is low remains an open problem. 

\begin{figure}
\centering
\begin{minipage}{.5\textwidth}
  \centering
  \includegraphics[width=0.8\linewidth]{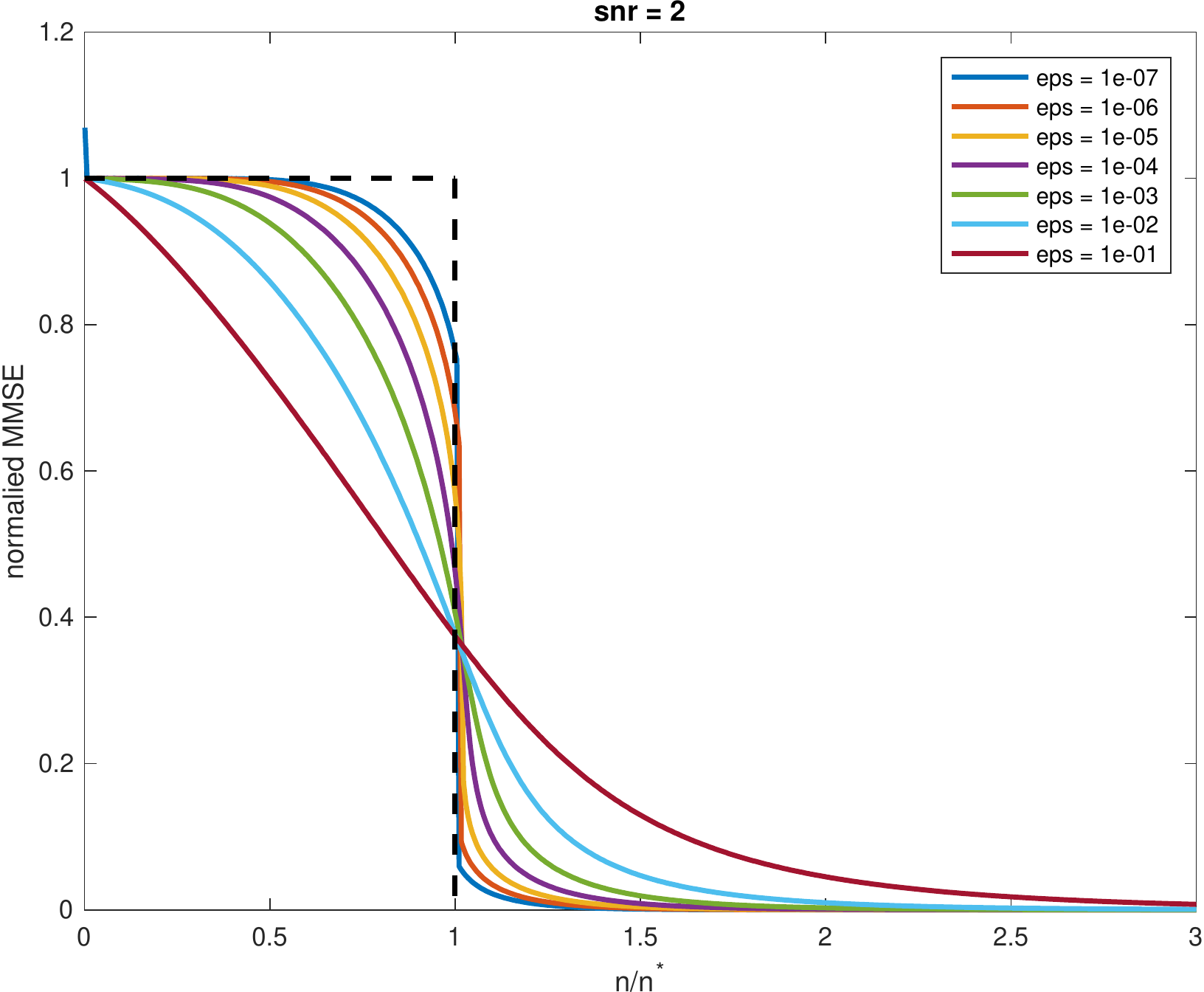}

  \label{fig:test1}
\end{minipage}%
\begin{minipage}{.5\textwidth}
  \centering
  \includegraphics[width=0.8\linewidth]{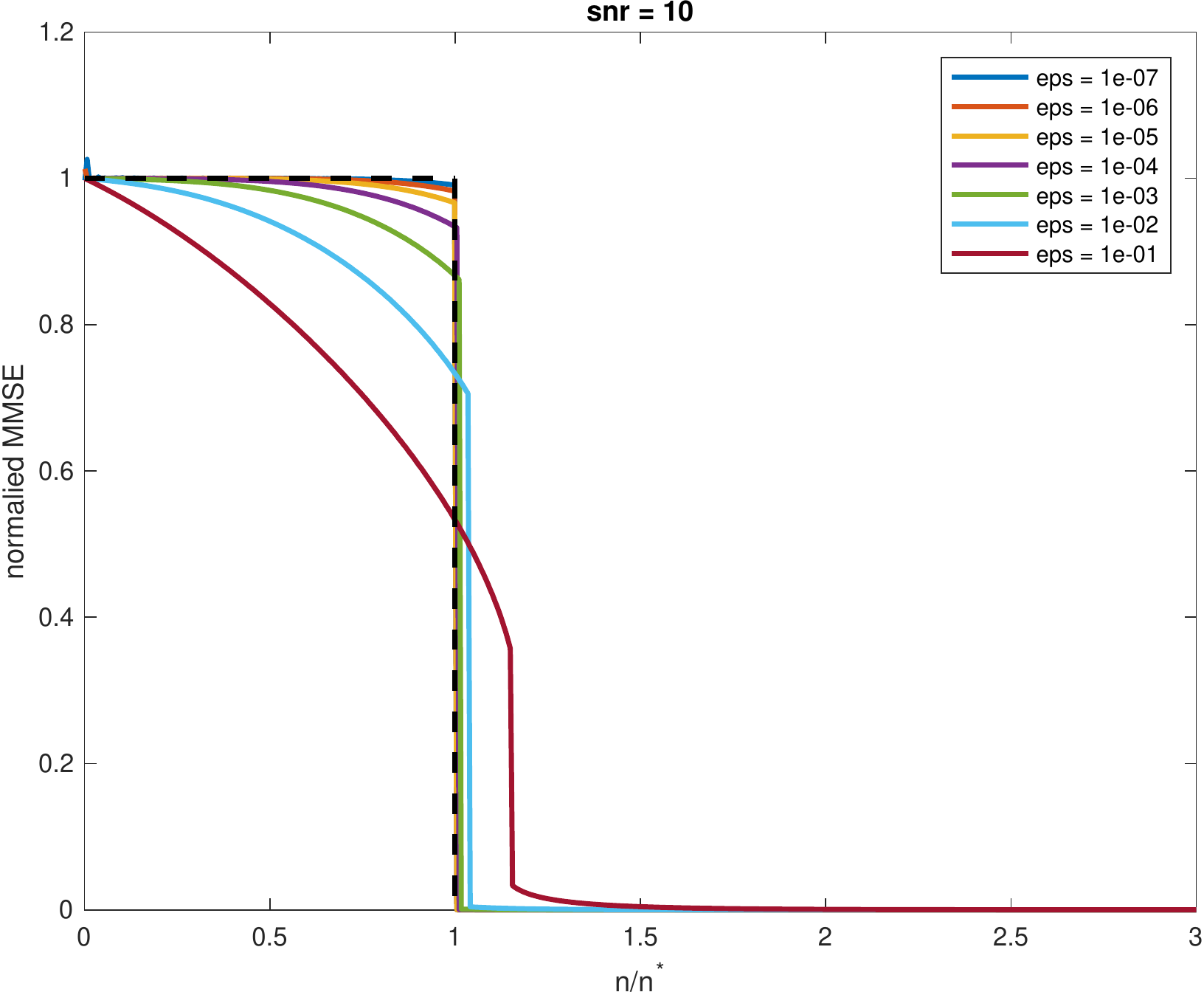}
  \label{fig:test2}
\end{minipage}\
 \caption{The limit of the replica-symmetric predicted MMSE $\mathcal M_{\eps,\gamma}(\cdot)$ as $\epsilon \to 0$ for signal to noise ratio (snr) $\gamma$ equal to $2$ (left curve) and equal to $10$ (right curve).}
\label{fig:replica}
\end{figure}

\paragraph*{Sparse Superposition Codes}

Constructing an algorithm for recovering a binary $k$-sparse $\beta$ from  $(Y=X\beta+W,X)$ receives a lot of attention from a coding theory point of view. The reason is that such recovery corresponds naturally to a code for the memoryless additive Gaussian white noise (AWGN) channel with signal-to-noise ratio equal to $k/\sigma^2$. Specifically in this context achieving strong recovery of a uniformly chosen binary $k$-sparse $\beta$ with $(1+\epsilon)n^*$ samples, for arbitrary $\epsilon>0$, corresponds exactly to capacity-achieving encoding-decoding mechanism of $\binom{p}{k}\sim (pe/k)^k$ messages through a AWGN channel. A recent line of work has analyzed a similar mechanism where $(p/k)^k$ messages are encoded through $k$-block-sparse vectors; that is the vector $\beta$ is designed to have at most one non-zero value in each of $k$ block of entries indexed by $i\floor{p/k},i\floor{p/k}+1,\cdots,(i+1)\floor{p/k} -1$ for $i=0,1,2,\ldots,k-1$. It has shown that by using various polynomial-time decoding mechanisms, such as adaptive successive decoding \cite{joseph:2012}, \cite{joseph:2014}, a soft-decision iterative decoder \cite{Barron:2012}, \cite{CHO:2014} and finally Approximate Message Passing techniques \cite{Rush:2017}, one can strongly recover the hidden $k$-block-sparse vector with $(1+\epsilon)n^*$ samples and achieve capacity. Their techniques are tailored to work for any $k=p^{1-c}$ with $c \in (0,1)$ and also require the vector to have carefully chosen non-zero entries, that is the hidden vector is not assumed to simply be binary. In this work \prettyref{thm:positive} establishes that under the simple assumption on $\beta$ being binary and arbitrarily (not block) $k$-sparse it suffices to make strong recovery possible with $(1+\epsilon)n^*$ samples when $k=o(p)$. Nevertheless, our decoding mechanism requires a search over the space of $k$-sparse binary vectors and therefore is not in principle polynomial-time. The design of a polynomial-time recovery algorithm for this task and $(1+\epsilon)n^*$ samples remains largely an open problem (see \cite{gamarnikzadik}).

\paragraph*{Information-theoretic limits up to constant factors for exact recovery} 
Although exact recovery is not our focus, we briefly mention some of the rich literature on the information-theoretic limits for the exact recovery of $\beta$, \ie, $\prob{\hat{\beta}=\beta} \to 1$ as $p \to \infty$ (see, \eg~\cite{wainwright:2009,fletcher:2009, Rad2011, wang2010information,ndaoud2018optimal}
and the references therein). Clearly since exact recovery implies weak and strong recovery, the sample sizes required to be achieve exact recovery are in principle no smaller than $n^*$.

Specifically, it has been shown in \cite[Theorem 1]{wainwright:2009} that
the maximum likelihood estimator achieves exact recovery if $ n \ge \Omega \left( \log \binom{p-k}{k} + \sigma^2 \log (p-k) \right)$
and $n-k \to +\infty$. Conversely, 
$n > \max \{ f_1(p,k), \ldots, f_k(p,k), k\}$ is shown in \cite[Theorem 1]{wang2010information} to be necessary for exact recovery, where 
$	 f_m(p,k)= 2 \frac{\log \binom{p-k+m}{m} -1 }{ \log \left( 1+  \frac{m(p-k)}{p-k+m}/\sigma^2 \right)}.
$ In the special regime where 
$k$ and $\sigma$ are fixed constants, it has been shown in \cite[Theorem 1]{jin2011limits} that exact recovery is information-theoretically possible if and only if $n \ge (1+o(1)) n^*$. Notice that this result achieves exact recovery for approximately $n^*$ sample size, but in this case of constant $k$ it can be easily seen that the two notions of exact and strong recovery coincide. 

Computationally, it has been shown in \cite[Section IV-B]{wainwright2009sharp} that LASSO achieves exact recovery in polynomial-time if $n \ge 2k \log (p-k)$.  More recently, it is shown in  \cite[Theorem 3.2, Corollary 3.2]{ndaoud2018optimal} that exact recovery can be achieved in polynomial-time, provided that $k=o(p)$, $\sigma \ge \sqrt{3}$, and $n \ge \Omega \left( k \log \frac{ep}{k} + \sigma^2 \log p  \right).$

\subsection{Proof Techniques}

In this section, we give an overview of our proof techniques.  
Given two probability distributions $P,Q$ with $P$ absolutely continuous to $Q$ and any convex function $f$ such that $f(1) = 0$, the $f$-divergence of $Q$ from $P$ is given by
$$ 
D_f(P\|Q) \triangleq \Exp_Q \left[ f \left(  \frac{d P}{ d Q} \right)  \right].
$$
Three choices of $f$ are of particular interests (See~\cite[Section 6]{PW-it} for details): 
\begin{itemize}
\item The \textit{Total Variation distance} $\mathrm{TV}(P,Q)$:  $f(x)=|x-1|/2$;
\item The \textit{Kullback-Leibler divergence} (a.k.a.\ relative entropy) $D( P \|Q)$ : $f(x) =x \log x$;
\item The \textit{$\chi^2$-divergence} $\chi^2(P\|Q)$: $f(x)=(x-1)^2$.
\end{itemize}
Note that the $\chi^2$-divergence $\chi^2(P\|Q)$ is equal to the variance 
of the Radon-Nikodym derivative (likelihood ratio) $dP/dQ$ under $Q$ and hence
$$
\chi^2(P\|Q) +1 = \Exp_Q \left[  \left ( \frac{dP}{dQ} \right) ^2  \right]
=\Exp_P \left[  \frac{dP}{dQ}   \right]. 
$$

A key to our proof is the following chain of inequalities:
\begin{align}
\mathrm{TV}(P,Q) \le  \sqrt{ 2 D( P ||Q )}  \le \sqrt{ 2 \log \left(  \chi^2 (P\|Q) +1 \right)},
\label{eq:TV_bound}
\end{align}
where the first inequality is simply Pinsker's inequality, and the second inequality holds by Jensen's inequality:
 \begin{align}
 D (P \| Q)  = \Exp_{P} \left[  \log \frac{ d P}{d Q} \right] 
 \le \log \left( \Exp_{ P} \left[ \frac{d P}{d Q}  \right] \right) 
 = \log \left(  \chi^2 (P\|Q) +1 \right). \label{eq:secondmoment_KL}
 \end{align}

Recall that to show the weak detection between $P$ and $Q_{\lambda}$ is impossible, it is equivalent to proving that 
$\mathrm{TV}\left(P,Q_{\lambda}\right) = o(1)$. In view of \prettyref{eq:TV_bound} there is a natural strategy towards proving it:
it suffices to prove that $\chi^2(P,Q_{\lambda})=o\left(1\right)$, which amounts to showing the 
second moment  $\Exp_Q\left[ (dP/dQ_{\lambda})^2\right]=1+o\left(1\right)$.
We prove that indeed if $n \le \left(1-o(1)\right) n^*/2$ and $\lambda$ is appropriately chosen,
then this second moment is indeed $1+o(1)$ (Theorem \ref{firstThm}); however, if $n> n^*/2$, then it blows up to infinity. 
This is because even if potentially $\mathrm{TV}(P,Q_{\lambda})=o(1)$, rare events can cause the second moment to explode and in particular (\ref{eq:TV_bound}) is far from being tight.

We are able to circumvent this difficulty by computing the second moment conditioned on an event $\calE$, 
which rules out the catastrophic rare ones. In particular, we introduce the following conditioned planted model.

\begin{definition}[Conditioned planted model]
Given a subset $\calE \subset \reals^{n \times p} \times \reals^{ p}$, 
define the conditioned planted model 

\begin{align}
P_\calE (X, Y ) = \frac{ \Exp_{\beta } \left[ P (X, Y \mid \beta) \indc{\calE}(X, \beta) \right]}{\prob{ \calE }} .  \label{eq:planted_cond}
\end{align}
\end{definition}

 Using this notation we can write
 \[
 P(X, Y) = (1- \eps) P_{\calE}(X, Y) + \eps P_{\calE^c}(X, Y), 
 \]
 where $\calE^c$ denotes the complement of $\calE$ and $\eps =  \prob{ (X,\beta) \in \calE^c}$. 
By Jensen's inequality and the convexity of KL-divergence,
\begin{align}
D(P ||Q_{\lambda}) \le (1- \eps)  D( P_{\calE} ||Q_{\lambda})  + \eps D( P_{\calE^c} ||Q_{\lambda}) . \label{eq:DPQ_convex_UB}
\end{align}

Under an appropriately chosen $\calE$, and $\lambda>0$, our main impossibility of detection result (Theorem \ref{thm:main}) shows that if $n \le (1+o(1)) n^*$, then
$\Exp_{Q_{\lambda}}[ (dP_\calE/dQ_{\lambda})^2]=1+o(1)$, or equivalently, $\chi^2(P_\calE \|Q_{\lambda}) =o(1)$, which immediately
implies that $D(P_\calE \|Q_{\lambda})=o(1)$ and $\mathrm{TV}(P_\calE, Q_{\lambda}) =o(1)$.
Finally, we argue that $\eps$ converges to $0$ sufficiently fast so that according to 
\prettyref{eq:DPQ_convex_UB}, $\TV(P,Q_{\lambda}) \le \TV(P_\calE ,Q) +o(1) =o(1)$ and 
$D(P \| Q_{\lambda}) \le D(P_\calE \|Q_{\lambda}) +o(1)=o(1).$

We remark that this (conditional) second moment method for providing detection lower bound has been used in many high-dimensional inference problems (see e.g.~\cite{Mossel12,Banks16,banks-etal-colt,PerryWeinBandeira16,wu2018statistical} and references therein). 

To further show weak recovery is impossible in the regime for sample size $n<n^*$ (Theorem \ref{thm:recovery}), we establish a lower bound of MSE
in terms of $D(P\|Q_{\lambda})$ (Lemma \ref{lem:MSE_LB})  which implies that the minimum MSE needs to be $\left(1-o(1)\right)k$ if $D(P\|Q_{\lambda})=o(n)$. The key underpinning our lower bound proof is the area theorem~\cite{measson2008maxwell,kudekar2017reed}.

\subsection{Notation and Organization}

Denote the identity matrix by $\identity$.
We let $\|X\|$ denote the spectral norm of a matrix $X$
and $\|x\|$ denote the $\ell_2$ norm of a vector $x$.
For any positive integer $n$, let $[n]=\{1, \ldots, n\}$.
For any set $T \subset [n]$, let $|T|$ denote its cardinality and $T^c$ denote its complement.
We use standard big $O$ notations,
e.g., for any sequences $\{a_p\}$ and $\{b_p\}$, $a_p=\Theta(b_p)$
if there is an absolute constant $c>0$ such that $1/c\le a_p/ b_p \le c$; $a_p =\Omega(b_p)$ or $b_p = O(a_p)$ if there exists  an absolute constant $c>0$ such that $a_p/b_p \ge c$. 
We say a sequence of events $\calE_p$ indexed by a positive integer $p$ 
holds with high probability, if the probability of $\calE_p$ converges to $1$ as $p \to +\infty$.
Without further specification, all the asymptotics are taken with respect to $p \to \infty$. 
All logarithms are natural and we use the convention $0 \log 0=0$.
For two real numbers $a$ and $b$, we use $a \vee b = \max\{a, b\}$ 
to denote the larger of $a$ and $b$. 
For two vectors $u, v$ of the same dimension, we use 
$\iprod{u}{v}$ denote their inner product. 
We use $\chi^2_n$ denote the standard chi-squared distribution 
with $n$ degrees of freedom. For $n,m,k \in \mathbb{N}$ with $m \leq k \leq n$ and  $m+k \leq n$ we denote by $\mathrm{Hyp}\left(n,m,k\right)$ the Hypergeometric distribution with parameters $n,m,k$ and probability mass function $p(s)=\binom{m}{s}\binom{n-m}{k-s}/ \binom{n}{k}, s \in [0,m] \cap \mathbb{Z}$.

The remainder of the paper is organized as follows.   
\prettyref{sec:main} presents the main results without proofs. 
\prettyref{sec:negative_dection}  and 
\prettyref{sec:negative_recovery} prove the negative results for detection and recovery,  respectively.
\prettyref{sec:positive} proves the positive results for detection and recovery. 
We conclude the paper in \prettyref{sec:open}, mentioning a few open problems.
Auxiliary lemmata and miscellaneous details are left to appendices.

\section{Main Results}\label{sec:main}

In this section we present our main results. The proofs are deferred to the following sections. 

\subsection{Impossibility of Weak Detection with $n<n^*$}
Our first impossibility detection result is based on a direct calculation of the second moment  between the planted model $P$ and the null model $Q_{\lambda}$. Specifically, we are able to show that weak detection between the two models is impossible, if $n \le (1-\alpha) n^*/2$ for some $\alpha=o_p(1)$ and $\lambda=\sqrt{k/\sigma^2+1}$.

\begin{theorem}\label{firstThm}

Suppose $k \le p^{1/2-\delta}$ for a fixed constant $\delta>0$ and 
$k/\sigma^2 \ge C$ for a sufficiently large constant $C$ only depending on $\delta$. 

If 
\begin{align}
n \leq \frac{1}{2} \left(1-  \frac{ \log \log \left( p/k \right) }{ \log \left( p/k \right) } \right)n^*, \label{eq:n_condition_small}
\end{align}
then for $\lambda_0=\sqrt{k/\sigma^2+1}$, it holds that 
$$
\chi^2 (P \| Q_{\lambda_0}) = o(1)
$$
Furthermore, $D(P\|Q_{\lambda_0})=o(1)$ and 
$
\mathrm{TV}(P,Q_{\lambda_0})=o(1).
$
\end{theorem}

The complete proof of the above Theorem can be found in Section \ref{NegD}. Nevertheless, let us provide here a short proof sketch. Using an explicit calculation, we first find that for any $\lambda>\sqrt{k/\sigma^2+1/2}$, 
$$\chi^2 \left(P \| Q_{\lambda}\right) 
= \lambda^{2n} \Exp_{S \sim \Hyp(p,k,k)} \left[ \left( 2\lambda^2 -1 - \frac{k+S}{\sigma^2} \right)^{-n/2} \left( 1 + \frac{k-S}{\sigma^2} \right)^{-n/2} \right]-1
$$ 
where $S=\iprod{\beta}{\beta'}$ is the overlap between two independent copies $\beta, 
\beta'$ and follows a Hypergeometric distribution with parameters $(p,k,k)$.
Plugging in $\lambda=\lambda_0=\sqrt{k/\sigma^2 +1}$, we get that 
\begin{align*}
\chi^2(P \| Q_{\lambda_0})=\Exp_{S \sim \Hyp(p,k,k)} 
\left[  \left( 1- \frac{S}{k+\sigma^2} \right)^{-n}  \right ]-1. 
\end{align*}
Using this we show that if $n \le \left(1+o(1) \right) n^*/2$,
then  $\chi^2(P \| Q_{\lambda_0})$ is indeed $o\left(1\right)$, implying by \prettyref{eq:TV_bound} the impossibility result. However, if $n> n^*/2$, then 
this $\chi^2$-divergence can be proven to blow up to infinity, rendering the method based on \prettyref{eq:TV_bound} uninformative in this regime. To see this, by considering 
the event $S=k$ which happens with probability $1/\binom{p}{k}$, we get that 
\begin{align}
\chi^2(P \| Q_{\lambda_0}) \ge \frac{1}{\binom{p}{k}} 
\left[  \left( 1- \frac{k}{k+\sigma^2} \right)^{-n}  \right ]-1
= \exp \left(  n \log \left( 1+ \frac{k}{\sigma^2} \right) - \log \binom{p}{k}  \right) -1.
\label{eq:chi_squared_blow_up}
\end{align}
Recall that $n^*$ is asymptotically equal to $2\log \binom{p}{k}/\log \left( 1+ \frac{k}{\sigma^2} \right)$. Hence if $n \ge n^*(1+\epsilon)/2$ for some constant $\epsilon>0$, then $\chi^2(P \| Q_{\lambda_0})\to +\infty$.

To be able to obtain tighter results and go all the way to $n^*$ sample size,  we resort to a
\emph{conditional} second moment method as explained in the proof techniques. Specifically we show that weak detection
is impossible for any $n \le (1-\alpha) n^*$, for some $\alpha>0$ that can be made to be arbitrarily small by increasing
$k/\sigma^2$ and $p/k$. In particular, this improves on the direct calculation of the $\chi^2$ distance  by a multiplicative factor of 2  and shows that $n^*$ is a sharp information theoretic threshold for weak detection between the planted model $P$ and the null model $Q_{\lambda_0}$.

Before formally stating our main theorem, 
we specify the conditioning event $\calE_{\gamma,\tau}$ which will be 
shown to hold with high probability in \prettyref{lmm:conditioning}
under appropriate choices of $\gamma$ and $\tau$. 

\begin{definition}[Conditioning event]
Given $\gamma \ge 0$ and $\tau \in [0,k]$, define 
an event $\calE_{\gamma, \tau} \subset \reals^{n \times p} \times \reals^{ p}$ as
\begin{align}
\calE_{\gamma, \tau}  =\left \{ \left( X, \beta \right) \, : \,    
\frac{ \| X (\beta + \beta') \|^2}{ \expect{ \| X (\beta + \beta') \|^2}} 
 \le 2 + \gamma  , \; \text{$\forall \beta' \in \{ 0, 1\}^p$ with $\|\beta'\|_0= k$ and $\iprod{\beta'}{\beta} \ge  \tau$} \right\}.  \label{eq:Egamma}
\end{align}

\end{definition}

To understand the value of $\gamma,\tau$ in the definition of this event, notice that for each $\beta,\beta'$, from the definition of $X$, we have 
$X (\beta + \beta') \sim \calN\left(0, 2(k+s)\identity_n \right)$ and 
therefore, 
$$
\frac{ \| X (\beta + \beta') \|^2}{ 2(k+s)} \sim \chi^2_n.
$$ 
Thus, by the concentration inequality of chi-squared distributions, the random variable $\frac{ \| X (\beta + \beta') \|^2}{ \expect{ \| X (\beta + \beta') \|^2}} $ is expected to concentrate around $1$ and thus is likely to be smaller than $2+\gamma$ for a relatively large 
$\gamma$. The parameter $\tau$ quantifies the set of $k$-sparse $\beta'$ that we expect this relation to hold. Notice that $\iprod{\beta'}{\beta} \ge  \tau$ is equivalent with the Hamming-distance between $\beta$ and $\beta'$ to be equal to $2\left(k-\tau \right)$. 

Next, we explain the intuition behind our choice of conditioning event 
$\calE_{\gamma, \tau}$. Recall that in view of \prettyref{eq:chi_squared_blow_up}, $
\chi^2(P \| Q_{\lambda_0})$ blows up to infinity when the overlap $\iprod{\beta}{\beta'}$ is equal to $k$. 
In fact, when the overlap $\iprod{\beta}{\beta'} =k$, $\| X (\beta + \beta') \|^2$ can be
enormously large, causing 
$\chi^2(P \| Q_{\lambda_0})$ to explode. 
We rule out this catastrophic event
by conditioning on $\calE_{\gamma, \tau}$ which
upper bounds $\| X (\beta + \beta') \|^2$ when the overlap $\iprod{\beta}{\beta'}$
is large (See \prettyref{eq:condition_key_step} for the key step of 
upper bounding $\| X (\beta + \beta') \|^2$).

As a result, we are able to prove that the $\chi^2$-divergence 
between the conditional planted model $P_{\calE_{\gamma, \tau}}$ and the null model $Q_{\lambda_0}$ for $\lambda_0=\sqrt{k/\sigma^2+1}$ is $o(1)$, which implies the following general impossibility of detection result.
\begin{theorem}\label{thm:main}
Suppose $ k \le p^{\frac{1}{2}-\delta} $ for an arbitrarily small fixed constant $\delta \in (0,\frac{1}{2})$ and 
$k/\sigma^2 \ge C$ for a sufficiently large constant $C$ only depending on $\delta$. 
Assume
$n \le \left( 1 - \alpha \right) n^*$ for $\alpha \in (0,1/2]$ such that
\begin{align}
\alpha = \frac{8}{ \log  ( 1+ k/\sigma^2) } \vee \frac{ 32 \log \log (p/k)}{\log (p/k)}. \label{eq:def_alpha_large}
\end{align}
Set 
$$
 \gamma= \frac{\alpha k \log (p/k) }{n}  \quad \text{and} \quad \tau=k\left(1-\frac{1}{\log^2 ( 1+ k/\sigma^2) }\right).
$$
Then for $\lambda_0=\sqrt{\frac{k}{\sigma^2}+1}$, 
\begin{align}
\chi^2 \left( P_{\calE_{\gamma, \tau}} \| Q_{\lambda_0} \right) = o(1).
\label{eq:cond_moment_bounded}
\end{align}
Furthermore $D(P_{\calE_{\gamma, \tau}}\| Q_{\lambda_0}) = o(1)$, $\mathrm{TV}(P_{\calE_{\gamma, \tau}},Q_{\lambda_0}) = o(1)$, and
$
\mathrm{TV}(P,Q_{\lambda_0})=o(1).
$
\end{theorem}

The proof of the Theorem can be found in Section \ref{NegD2}.

\subsection{Impossibility of Weak Recovery with $n<n^*$}
In this section we present our impossibility of recovery result. We do this using the impossibility of detection  result established above. Specifically we first strengthen Theorem \ref{thm:main} and show that under the assumptions of \prettyref{thm:main}, 
$D(P\|Q_{\lambda_0})=o_p(1)$. Notice that this is not needed to conclude impossibility of detection, that is $TV(P,Q_{\lambda_0})=o(1)$, but is needed here for establishing the impossibility of recovery result. As a second step, inspired by the celebrated area theorem,
we establish (Lemma \ref{lem:MSE_LB}) a lower bound to the minimum MSE
in terms of $D(P\|Q_{\lambda_0})$, which is potentially of independent interest. The lemma essentially quantifies the natural idea that if the data $(Y,X)$ drawn from planted model are statistically close to the data $(Y,X)$ drawn from null model then there are limitations on the performance of recovering the hidden vector $\beta$ based on the data $(Y,X)$ from the planted model. Interestingly the lemma itself does not require the hidden vector $\beta$ to be binary or $k$-sparse but only to satisfy $\mathbb{E}\left[\|\beta\|_2^2\right]=k$.
Combining the two steps allows us to conclude that the minimum MSE is $k(1+o_p(1))$; hence the impossibility of 
weak recovery.

\begin{theorem}\label{thm:recovery}
Suppose $ k \le p^{\frac{1}{2}-\delta} $ for an arbitrarily small fixed constant $\delta \in (0,\frac{1}{2})$ and 
$k/\sigma^2 \ge C$ for a sufficiently large constant $C$ only depending on $\delta$. Let $\lambda_0=\sqrt{k/\sigma^2+1}$. 
If
$n \le \left( 1 - \alpha \right) n^*$ for $\alpha \in (0,1/2]$ given in \prettyref{eq:def_alpha_large}, 
then it holds that 
\begin{align}
D \left( P \| Q_{\lambda_0} \right)=o_p(1). 
\label{eq:KL_bounded}
\end{align}
Furthermore, if $n \le \floor{ (1 - \alpha ) n^*}-1$, then for any estimator $\hat{\beta}$ that is a function of $X$ and $Y$, 
\begin{align}
\MSE \left( \hat{\beta} \right)= k \left( 1+o_p(1) \right). 
\label{eq:MSE_bound}
\end{align}
\end{theorem}

The proof of the above Theorem can be found in Section \ref{NegR}.

\subsection{Positive Result for Strong Recovery with $n>n^*$}
This subsection and the next one are in the regime where $n>n^*$. In these regimes, in contrast to $n<n^*$ we establish that both strong recovery and strong detection are possible.

Towards recovering the vector $\beta$, we consider the Maximum Likelihood Estimator (MLE) of $\beta$:
$$
\hat{ \beta }=\arg \min_{\beta' \in \{0,1\}^p, \|\beta'\|_0=k}  \| Y-X\beta'  \|^2.
$$We show that MLE achieves strong recovery of $\beta$ if $ n \ge (1+\epsilon) n^*$ for an arbitrarily small but fixed
constant $\epsilon$ whenever $k=o(p)$ and $k/\sigma^2 \ge C(\epsilon)$ 
for a sufficiently large constant $C\left(\epsilon\right)>0$.

Specifically, we establish the following result.
 \begin{theorem} \label{thm:positive} 
 Suppose $\log \log \left( p/k \right) \geq 1$. If 
\begin{align}
n \ge  \left(1+\frac{ \log 2 }{ \log \left( 1+  k/(2\sigma^2) \right) }  \right)  \left( 1 + \frac{4 \log \log (p/k)}{\log (p/k)} \right) n^*,  \label{eq:recovery_cond}
\end{align}
then 
\begin{align}
\prob{ \| \hat{ \beta} - \beta\|^2 \ge \frac{2k}{\log (p/k)} }
\le \frac{e^2}{\log^2 (p/k)(1-e^{-1})}.  \label{eq:recovery_positive}
\end{align}
Furthermore, if additionally $k=o(p)$, then 
 \begin{align}
 \frac{1}{k} \expect{\norm{ \hat{\beta} - \beta}^2 }  = o_p(1), \label{eq:strong_recovery_positive}
 \end{align}
 \ie,  MLE achieves strong recovery of $\beta$. 
\end{theorem}
The proof of the above Theorem can be found in Section \ref{PosR}.
\subsection{Positive Result for Strong Detection with $n>n^*$}
In this subsection we establish that when $n>n^*$ strong detection is possible. To distinguish the planted model $P$ and the null model $Q_\lambda$, we consider the test statistic:
\[
\calT(X,Y) = \min_{\beta' \in \{0,1\}^p, \| \beta'\|_0 =k } \frac{  \| Y - X \beta\|^2}{ \|Y\|^2}.
\]

\begin{theorem}\label{thm:posDetection}
Suppose 
\begin{align}
\log n - \frac{2}{n} \log \binom{p}{k} \to +\infty \label{eq:ass_detection_cond}
\end{align}
and
\begin{align}
n \ge \frac{2 \log \binom{p}{k} }{ \log \left(1+k/\sigma^2 \right) + \log (1-\alpha) }
\label{eq:ass_detection_cond2}
\end{align}
for an arbitrarily small but fixed constant $\alpha \in (0,1)$. Then by letting 
$\tau=\frac{1}{(1-\alpha/2)(1+k/\sigma^2)}$, we have that
$$
P\left( \calT(X,Y) \ge \tau \right) + Q_{\lambda} \left( \calT(X,Y) \le \tau \right) =o(1),
$$
which achieves the strong detection between the planted model $P$ and the null model $Q_\lambda$.

\end{theorem}The proof of Theorem \ref{thm:posDetection} can be found in Section \ref{posdet}.

We close this section with one remark, explaining the newly introduced condition (\ref{eq:ass_detection_cond}).
\begin{remark}
Recall that $n^*=2k\log (p/k) / \log (1+k/\sigma^2)$ and $\binom{p}{k} \le (ep/k)^k$. Thus,
\begin{align*}
\log n^* - \frac{2}{n^*} \log \binom{p}{k}
& \ge \log \left( \frac{2k\log (p/k)}{ \log (1+k/\sigma^2)} \right)
- \frac{\log (ep/k)}{\log (p/k)} \log \left(1+k/\sigma^2 \right) \\
& \ge  \log \left(k  \log \frac{p}{k} \right) - \log \log  \left(1+k/\sigma^2 \right)
-\log \left(1+k/\sigma^2 \right) - \frac{\log \left(1+k/\sigma^2 \right)}{\log (p/k)}.
\end{align*}
If $1+ k/\sigma^2  \leq \left(k  \log \frac{p}{k} \right)^{1-\eta}$
for some fixed constant $\eta>0$,
then it follows from the last displayed equation that 
$$
\log n^* - \frac{2}{n^*} \log \binom{p}{k}
\ge \eta \log \left(k  \log \frac{p}{k} \right) - \log \log \left(k  \log \frac{p}{k} \right)
-\frac{\log \left(k  \log \frac{p}{k} \right)}{\log (p/k)}
$$
which goes to $+\infty$ as $p \to +\infty$; hence 
$n^*$ satisfies (\ref{eq:ass_detection_cond}).

Therefore, assuming that $1+ k/\sigma^2  \leq \left(k  \log \frac{p}{k} \right)^{1-\eta}$
and $n \geq (1+\epsilon)n^*$ for some arbitrarily small constants $\eta, \epsilon>0$, 
there exists a constant $C=C(\epsilon)>0$ such that if $k/\sigma^2 \geq C(\epsilon)$, then the test statistic $\calT(X,Y)$ achieves strong detection.
\end{remark}

\section{Proof of Negative Results for Detection}
\label{sec:negative_dection}

\subsection{Proof of Theorem \ref{firstThm}} \label{NegD}
We start with an explicit computation of the chi-squared divergence $\chi^2 (P \| Q_{\lambda})$.

\begin{proposition}\label{chisq}
For any $\lambda > \sqrt{k/\sigma^2 + 1/2}$, 
$$\chi^2 (P \| Q_{\lambda}) 
= \lambda^{2n} \Exp_{S \sim \Hyp(p,k,k)} \left[ \left( 2\lambda^2 -1 - \frac{k+S}{\sigma^2} \right)^{-n/2} \left( 1 + \frac{k-S}{\sigma^2} \right)^{-n/2} \right]-1.$$

\end{proposition} 
\begin{proof}
Since the marginal distribution of $X$ is the same under the planted and null models, it follows that for any $\beta$,
$$
\frac{ P (X, Y)}{ Q_{\lambda}(X,Y)} =  \frac{ P(Y|X) }{Q_{\lambda}(Y)} = \frac{\Exp_{\beta} [ P(Y|X, \beta)]}{Q_{\lambda}(Y)}. 
$$
Therefore
$$
\left( \frac{ P (X, Y)}{ Q_{\lambda}(X,Y)} \right)^2 = 
\Exp_{\beta \independent \beta'} \left[  \frac{P(Y|X, \beta) P(Y|X, \beta')  }{Q_{\lambda}^2(Y)}\right],
$$
where $\beta \independent \beta'$ denote two independent copies.
By Fubini's theorem, we have
\begin{align}
\Exp_{Q_{\lambda}} \left[ \left( \frac{ P }{ Q_{\lambda}} \right)^2 \right] = 
\Exp_{\beta \independent \beta'} \Exp_{X }  \Exp_{Y } \left[   \frac{P(Y|X, \beta) P(Y|X, \beta')  }{Q_{\lambda}^2(Y)}  \right],
\label{eq:moment_fubini}
\end{align}
where $X_{ij} \iiddistr \calN(0,1)$ and $Y_{ij} \iiddistr \calN(0, \lambda^2 \sigma^2 ).$

Since in the planted model, conditional on $(X, \beta)$, $Y \sim \calN( X\beta, \sigma^2 \identity_n)$.
It follows that 
\begin{align*}
\frac{P(Y|X, \beta) }{Q_{\lambda}(Y)} & =  \lambda^n \exp \left(  - \frac{1}{2\sigma^2} \| Y- X \beta\|_2^2 + 
\frac{1}{2\lambda^2 \sigma^2} \| Y \|_2^2  \right) \\
& =\lambda^n \exp \left(  -  \frac{\lambda^2 -1 }{2\sigma^2 \lambda^2 } \| Y\|_2^2 
+ \frac{1}{\sigma^2} \iprod{Y}{X \beta} - \frac{1}{2\sigma^2} \| X \beta\|_2^2 \right).
\end{align*}
Hence,
\begin{align*}
& \frac{P(Y|X, \beta) P(Y|X, \beta')  }{Q_{\lambda}^2(Y)} \\
& = 
\lambda^{2n} \exp \left(  - \frac{\lambda^2 -1 }{\sigma^2 \lambda^2 } \| Y\|_2^2 
+ \frac{1}{\sigma^2} \iprod{Y}{X \left( \beta + \beta' \right) } - \frac{1}{2\sigma^2} 
\left( \| X \beta\|_2^2  + \|X \beta'\|_2^2 \right) \right) \\
& = \lambda^{2n} \exp \left(  -  \frac{\lambda^2 -1 }{\sigma^2 \lambda^2 }  
\left\|  Y -  \frac{\lambda^2 X \left(\beta+ \beta' \right)}{2(\lambda^2-1)}  \right\|_2^2 
 + \frac{\lambda^2 \left\|  X \left(\beta+ \beta' \right) \right\|_2^2}{4 (\lambda^2-1) \sigma^2 } 
 - \frac{1}{2\sigma^2} 
\left( \| X \beta\|_2^2  + \|X \beta'\|_2^2 \right) \right).
\end{align*}
Using the fact that $\mathbb{E}\left[e^{tZ^2}\right]=\frac{1}{\sqrt{1-2t\sigma^2 }}e^{\mu^2 t/ (1-2t\sigma^2)}$ 
for $t<1/2$ and $Z \sim \calN(\mu,\sigma^2)$, we get that
\begin{align*}
\Exp_Y 
\left[
\exp \left(  -  \frac{\lambda^2 -1 }{\sigma^2 \lambda^2 }  
\left\|  Y -  \frac{\lambda^2 X \left(\beta+ \beta' \right)}{2(\lambda^2-1)}  \right\|_2^2 \right)
\right] 
=\frac{1}{ (2\lambda^2-1)^{n/2} } \exp \left( - \frac{\lambda^2\left\|  X \left(\beta+ \beta' \right) \right\|_2^2 }{4(2\lambda^2-1)(\lambda^2-1) \sigma^2} \right).
\end{align*} 
Combining the last two displayed equations yields that 
\begin{align}
& \Exp_{Y} \left[   \frac{P(Y|X, \beta) P(Y|X, \beta')  }{Q_{\lambda}^2(Y)}  \right] \nonumber \\
& =\frac{\lambda^{2n}}{ (2\lambda^2-1)^{n/2} } 
\exp \left\{ \frac{1}{2\sigma^2(2\lambda^2-1)} 
\left( (1-\lambda^2) \left( \norm{X\beta}^2 + \norm{X \beta'}^2 \right) + 2\lambda^2 \iprod{X\beta}{X \beta'}  \right) \right\}. \label{eq:moment_exp_Y}
\end{align}
Let $T=\text{supp}(\beta)$ and $T'=\text{supp}(\beta')$. Let $X_i$ denote the $i$-th column of $X.$ Define
$$
Z_0 =\sum_{i \in T \cap T'} X_i, \quad Z_1= \sum_{i \in T \setminus T'} X_i, \quad Z_2 =\sum_{i \in T'\setminus T} X_i.
$$
Then conditional on $\beta$ and $\beta'$, $Z_0, Z_1, Z_2$ are mutually independent and 
$$
Z_0 \sim \calN( 0, s \identity_{n} ), \quad Z_1 \sim \calN(0, (k-s) \identity_n ), \quad Z_2 \sim \calN(0, (k-s)\identity_n),
$$
where $s= | T \cap T' | = \iprod{\beta}{\beta'}$.
Moreover, $X\beta,X\beta'$ can be expressed as a function of $Z_0,Z_1,Z_2$ simply by 
\begin{align}
X \beta = Z_0 + Z_1 \text{ and }X\beta' = Z_0 + Z_2.
\label{eq:ZXrel}
\end{align}

 Let $Z=[Z_0, Z_1, Z_2]^t \in \mathbb{R}^{3n}.$ Using (\ref{eq:moment_exp_Y}) and (\ref{eq:ZXrel}) and elementary algebra we have
\begin{align}
\Exp_{Y} \left[   \frac{P(Y|X, \beta) P(Y|X, \beta')  }{Q_{\lambda}^2(Y)}  \right]  
 =\frac{\lambda^{2n}}{ (2\lambda^2-1)^{n/2} }  \exp \left\{ t  Z^\top A  Z \right\},
 \label{eq:moment_exp_Y2}
\end{align}
where 
\begin{align*}
t = \frac{1}{2\sigma^2(2\lambda^2-1)}, \quad \text{ and }
A = \begin{bmatrix} 
2 & 1 & 1 \\
1 & 1-\lambda^2 & \lambda^2 \\
1 & \lambda^2 & 1-\lambda^2
\end{bmatrix}
\otimes \identity_{n} \in \mathbb{R}^{3n \times 3n},
\end{align*}
where by $A \otimes B $ we refer to the Kronecker product between two matrices $A$ and $B$.
Note that $Z$ is a zero-mean Gaussian vector with covariance matrix 
$$
V= \diag{s, k-s, k-s} \otimes \identity_n.
$$
Note that 
$$
AV = 
\left( \begin{bmatrix} 
2 & 1 & 1 \\
1 & 1-\lambda^2 & \lambda^2 \\
1 & \lambda^2 & 1-\lambda^2
\end{bmatrix} \diag{s, k-s, k-s}  \right) \otimes \identity_n.
$$
It is straightforward to find that the eigenvalues of $AV$ are $0$ of multiplicity $n$,
$k+s$ of multiplicity $n$, and $(k-s)(1-2\lambda^2)$ of multiplicity $n.$
Thus,
\begin{align}
\det(\identity_{3n} - 2 t A V ) = \left(1 - 2t (k+s) \right)^{n} \left(1 - 2t (k-s) (1-2\lambda^2) \right)^{n}.
\label{eq:moment_det}
\end{align}
It follows from \prettyref{eq:moment_exp_Y2} that 
\begin{align}
\Exp_X \Exp_{Y} \left[   \frac{P(Y|X, \beta) P(Y|X, \beta')  }{Q_{\lambda}^2(Y)}  \right] 
& =\frac{\lambda^{2n}}{ (2\lambda^2-1)^{n/2} }  \Exp_Z \left[  e^{ t Z^\top A Z }\right] \nonumber \\
& =\frac{\lambda^{2n}}{ (2\lambda^2-1)^{n/2} }  \frac{1}{\sqrt{\det(\identity_{3n} - 2 t A V ) }},\label{eq:moment_exp_XY}
\end{align}
where the last equality holds if $t< \frac{1}{2(k+s)}$ and 
follows from the expression of MGF of 
a quadratic form of normal random variables, see, \eg, \cite[Lemma 2]{baldessari1967distribution}. 

Combining \prettyref{eq:moment_det} and \prettyref{eq:moment_exp_XY}  yields that
if $t = \frac{1}{2\sigma^2(2\lambda^2-1)}< \frac{1}{2(k+s)}$,
\begin{align*}
\Exp_X \Exp_{Y} \left[   \frac{P(Y|X, \beta) P(Y|X, \beta')  }{Q_{\lambda}^2(Y)}  \right] 
& =\frac{\lambda^{2n}}{ (2\lambda^2-1)^{n/2} } \left(1 - \frac{k+s}{ \sigma^2 (2\lambda^2-1 )} \right)^{-n/2}
\left( 1 + \frac{k-s}{\sigma^2} \right)^{-n/2} \\
& = \lambda^{2n} \left( 2\lambda^2 -1 - \frac{k+s}{\sigma^2} \right)^{-n/2} \left( 1 + \frac{k-s}{\sigma^2} \right)^{-n/2}.
\end{align*}
Note that if 
$2\lambda^2 -1 > \frac{2k}{\sigma^2}$,
then 
$\frac{1}{2\sigma^2(2\lambda^2-1)}< \frac{1}{2(k+s)}$
for all $0 \le s \le k$.
It follows from \prettyref{eq:moment_fubini} that   
if $2\lambda^2 -1 > \frac{2k}{\sigma^2}$,
 then 
\begin{align*}
\Exp_{Q_{\lambda}} \left[ \left( \frac{ P }{ Q_{\lambda}} \right)^2 \right]
= \lambda^{2n} \Exp_{S \sim \Hyp(p,k,k)} \left[ \left( 2\lambda^2 -1 - \frac{k+S}{\sigma^2} \right)^{-n/2} \left( 1 + \frac{k-S}{\sigma^2} \right)^{-n/2} \right]. 
\end{align*}

\end{proof}

We establish also the following lemma. 

\begin{lemma}\label{lmm:MGF_boundMain}
Suppose $ k \le p^{\frac{1}{2}-\delta} $ for an arbitrarily small fixed constant $\delta \in (0,\frac{1}{2})$ and 
$\frac{k}{\sigma^2} \ge C$ for a sufficiently large constant $C$ only depending on $\delta$.
 If $n$ satisfies condition \prettyref{eq:n_condition_small}, then
\begin{align}
\Exp_{S \sim \Hyp(k,k,p)} \left[ \left( 1- \frac{S}{k+\sigma^2} \right)^{-n} \right ]
= 1+o_p(1). \label{eq:MGF_hyp_bound_1_Main}
\end{align}
\end{lemma}
\begin{proof}
The lemma readily follows by combining  \prettyref{lmm:MGF_bound} and 
\prettyref{lmm:MGF_large_regime} with $\alpha=\frac{\log \log (p/k) }{\log (p/k)}$ and $c=p^{-1/2-\delta}$. 
\end{proof}

\begin{proof}[Proof of Theorem \ref{firstThm}]
Using Proposition \ref{chisq} for $\lambda=\lambda_0$ satisfying $\lambda_0^2 = k/\sigma^2 +1$ we have
\begin{align*}
\chi^2(P \| Q_{\lambda_0})=\Exp_{S \sim \Hyp(p,k,k)} 
\left[  \left( 1- \frac{S}{k+\sigma^2} \right)^{-n}  \right ]-1. 
\end{align*}Using now \prettyref{lmm:MGF_boundMain} we have $\chi^2(P \| Q_{\lambda_0})=o(1)$. The chain of inequalities (\ref{eq:TV_bound}) concludes the proof of  Theorem \ref{firstThm}.

\end{proof}

\subsection{Proof of Theorem \ref{thm:main}} \label{NegD2}

\begin{proof}[Proof of Theorem \ref{thm:main}] 

 For notational simplicity we denote in this proof the probability measure $Q_{\lambda_0}$ simply by $Q$ and the event 
 $\calE_{\gamma, \tau}$ by $\calE.$

We first show that \prettyref{eq:cond_moment_bounded}
implies $D(P_\calE\| Q) = o(1)$, $\mathrm{TV}(P_\calE,Q) = o(1)$, and
$
\mathrm{TV}(P,Q)=o(1).
$

It follows from  \prettyref{eq:TV_bound} that  
$D(P_\calE\| Q) = o(1)$ and $\mathrm{TV}(P_{\calE} ,Q )=o(1).$ 
Observe that under our choice of $\tau$ and $\gamma$,  \prettyref{lmm:conditioning} implies that 
\begin{align}
\prob{\calE^c} \le \exp \left( - \frac{n\gamma}{8} \right) = \exp \left( -  \frac{\alpha k \log (p/k) }{8} \right)
\le \exp \left(  -  4 k \log \log (p/k) \right) =o_p(1). 
\label{eq:cond_event_prob_bound}
\end{align}
Thus, in view of 
\prettyref{eq:DPQ_convex_UB}, we get that 
\begin{align*}
 \mathrm{TV}(P,Q)  & \le \left(1-\prob{\calE^c} \right) \mathrm{TV}(P_{\calE} ,Q ) + \prob{\calE^c} \mathrm{TV}(P_{\calE^c} ,Q ) \\
 & \le \mathrm{TV}(P_{\calE} ,Q ) + \prob{\calE^c} = o(1).
\end{align*}

Next we prove \prettyref{eq:cond_moment_bounded}.  
We first carry calculations for any $\lambda>\sqrt{k/\sigma^2+1/2}$; we then restrict to $\lambda=\sqrt{k/\sigma^2+1}$.
In view of \prettyref{eq:planted_cond}, we have 
$$
\frac{P_\calE(X,Y)}{Q(X,Y)} =  
\frac{1}{Q(Y) Q(X)} 
\Exp_{\beta} \left[ \frac{  P(X) P (Y |X, \beta) \indc{\calE} (X, \beta) }{ \prob{ \calE  } } \right] 
= \Exp_{\beta}  \left[ \frac{  P (Y |X, \beta)  \indc{\calE} (X, \beta) } { Q(Y)  \prob{   \calE  } }\right],
$$
where the last equality holds because $P(X)=Q(X)$. 
Hence 
$$
\left( \frac{P_\calE(X,Y)}{Q(X,Y) } \right)^2
= \Exp_{\beta \independent \beta'} 
\left[ \frac{ P(Y|X, \beta) P(Y|X, \beta') 
\indc{ \calE  } (X, \beta) \indc{ \calE  } (X, \beta')  } 
{  Q^2(Y) \mathbb{P}^2\left\{ \calE  \right\} } \right],
$$
where $\beta'$ is an independent copy of $\beta$. 
Recall $\prob{ \calE }=1-o(1)$.
Therefore, 
$$
\Exp_{Q} \left[ \left( \frac{ P_\calE }{ Q} \right)^2 \right] = \left( 1+o(1) \right)
\Exp_{\beta \independent \beta'} 
\Exp_{X} 
\left[ \Exp_{Y}  \left[ \frac{ P(Y|X, \beta) P(Y|X, \beta')  } { Q^2(Y)  } \right] 
\indc{ \calE  } (X, \beta) \indc{ \calE  } (X, \beta')  
\right].
$$
It follows from \prettyref{eq:moment_exp_Y} that 
\begin{align*}
& \Exp_{Y} \left[   \frac{P(Y|X, \beta) P(Y|X, \beta')  }{Q^2(Y)}  \right]\\
& = \frac{\lambda^{2n}}{ (2\lambda^2-1)^{n/2} } 
\exp \left\{ \frac{ \| X ( \beta + \beta' )\|^2 
 - \left(2 \lambda^2 -1 \right) \| X(\beta -\beta')\|^2}{4\sigma^2(2\lambda^2-1)} 
 \right\}.
\end{align*}
Combining the last two displayed equation yields that
\begin{align}
& \Exp_{Q} \left[ \left( \frac{ P_\calE }{ Q} \right)^2 \right]  \nonumber \\ 
& =  \frac{\left( 1+o(1) \right)\lambda^{2n}}{ (2\lambda^2-1)^{n/2} } 
\Exp_{\beta \independent \beta'} 
\Exp_{X} \left[ e^{ \frac{ \| X ( \beta + \beta' )\|^2 
 - \left(2 \lambda^2 -1 \right) \| X(\beta -\beta')\|^2}{4\sigma^2(2\lambda^2-1)} } 
\indc{ \calE  } (X, \beta) \indc{ \calE  } (X, \beta') 
 \right]. \label{eq:moment_cond}
\end{align}

Next we break the right hand side of \prettyref{eq:moment_cond} 
into two disjoint parts depending on whether $\iprod{\beta}{\beta'} \le \tau$. We prove that the part where $\iprod{\beta}{\beta'} \le \tau$ is $1+o(1)$ and the part where $\iprod{\beta}{\beta'} > \tau$ is $o(1)$. Combining them we conclude the desired result.

{\bf Part 1}: Note that

\begin{align}
& \Exp_{X} \left[  \exp \left\{ \frac{ \| X ( \beta + \beta' )\|^2 
- \left(2 \lambda^2 -1 \right) \| X(\beta -\beta')\|^2}{4\sigma^2(2\lambda^2-1)} 
 \right\} 
 \indc{ \calE  } (X, \beta) \indc{ \calE  } (X, \beta')  \right]
 \indc{ \iprod{\beta}{\beta'} \le \tau} \nonumber \\
 & \le \Exp_{X} \left[  \exp \left\{ \frac{ \| X ( \beta + \beta' )\|^2 
 - \left(2 \lambda^2 -1 \right) \| X(\beta -\beta')\|^2}{4\sigma^2(2\lambda^2-1)} 
 \right\} \right] \indc{  \iprod{\beta}{\beta'} \le \tau }.
 \label{eq:part_1_moment_1}
\end{align}
Since $\iprod{\beta+\beta'}{\beta-\beta'}=0$ and $X_{ij} \iiddistr \calN(0,1)$, conditional on $(\beta, \beta')$,  $\mathrm{Cov}(X(\beta+\beta'),X(\beta-\beta'))=0$ and therefore $X(\beta+\beta') \sim \calN( 0, 2(k+s) \identity_{n})$  is independent of $X(\beta-\beta') \sim \calN( 0, 2(k-s) \identity_n) $, for $s=\iprod{\beta}{\beta'}$.
Therefore, 
\begin{align}
& \Exp_{X} \left[  \exp \left\{ \frac{ \| X ( \beta + \beta' )\|^2 
 - \left(2 \lambda^2 -1 \right) \| X(\beta -\beta')\|^2}{4\sigma^2(2\lambda^2-1)} 
 \right\} \right] \nonumber \\
 & = \Exp_{X} \left[ \exp \left\{ \frac{ \| X ( \beta + \beta' )\|^2 }{4\sigma^2(2\lambda^2-1)} 
 \right\} \right]
 \Exp_{X} \left[ \exp \left\{ - \frac{  \| X(\beta -\beta')\|^2}{4\sigma^2} 
 \right\} \right] \nonumber \\
 &  =  
 \left( 1- \frac{  (k+s) }{\sigma^2 (2\lambda^2-1)} \right)^{-n/2}  \left( 1+ \frac{  (k-s) }{\sigma^2} \right)^{-n/2},
 \label{eq:part_1_moment_2}
\end{align}
where the last equality holds if $\lambda>\sqrt{(k+s)/(2\sigma^2)+1/2}$
and follows from the fact that $\mathbb{E}_{Z \sim \chi^2(1)}\left[e^{-tZ}\right]=\frac{1}{\sqrt{1+2t}}$ 
for $t>-1/2$. Combining \prettyref{eq:part_1_moment_1} and \prettyref{eq:part_1_moment_2} yields that if $\lambda>\sqrt{k/\sigma^2+1/2}$, then
\begin{align}
& \frac{\lambda^{2n}}{ (2\lambda^2-1)^{n/2} } \Exp_{\beta \independent \beta'}  \Exp_{X} 
\left[  e^{ \frac{ \| X ( \beta + \beta' )\|^2 
 - \left(2 \lambda^2 -1 \right) \| X(\beta -\beta')\|^2}{4\sigma^2(2\lambda^2-1)} } \indc{ \calE  } (X, \beta) \indc{ \calE  } (X, \beta')  \right] \indc{  \iprod{\beta}{\beta'} \le \tau } \nonumber \\
 & \le 
\frac{\lambda^{2n}}{ (2\lambda^2-1)^{n/2} } \Exp_{\beta \independent \beta'} 
 \left[ 
 \left( 1- \frac{  (k+s) }{\sigma^2 (2\lambda^2-1)} \right)^{-n/2}  \left( 1+ \frac{  (k-s) }{\sigma^2} \right)^{-n/2}
 \indc{  s \le \tau }
 \right], 
 \nonumber 
 \end{align}
 In particular, by plugging in $\lambda=\sqrt{k/\sigma^2+1}$, we get that 
 \begin{align}
 & \frac{\lambda^{2n}}{ (2\lambda^2-1)^{n/2} } \Exp_{\beta \independent \beta'}  \Exp_{X} 
\left[  e^{ \frac{ \| X ( \beta + \beta' )\|^2 
 - \left(2 \lambda^2 -1 \right) \| X(\beta -\beta')\|^2}{4\sigma^2(2\lambda^2-1)} } \indc{ \calE  } (X, \beta) \indc{ \calE  } (X, \beta')  \right] \indc{  \iprod{\beta}{\beta'} \le \tau } \nonumber \\
 & \overset{(a)}{\le} \left(\frac{k}{\sigma^2}+1\right)^{n} \Exp_{S \sim \mathrm{Hyp}{(p,k,k)}} \left\{  \left( 1+ \frac{  (k-S) }{\sigma^2} \right)^{-n}  \indc{S \le \tau } \right\} \nonumber \\
&=\Exp_{S \sim \mathrm{Hyp}{(p,k,k)}} \left\{  \left( 1- \frac{ S }{k+\sigma^2} \right)^{-n}  \indc{S \le \tau } \right\} ,
\end{align}
where $(a)$ holds by noticing that $s=\iprod{\beta}{\beta'}$ follows an Hypergeometric distribution with parameters $(p,k,k)$ as the dot product of two uniformly at random chosen binary $k$-sparse vectors.

Using Lemma \ref{lmm:MGF_bound} we conclude that under our assumptions,  there exists a constant $C>0$ depending only on $\delta>0$ such that if $k/\sigma^2 \geq C$ then 
$$ 
\Exp_{S \sim \mathrm{Hyp}{(p,k,k)}} \left\{  \left( 1- \frac{ S }{k+\sigma^2} \right)^{-n}  \indc{S \le \tau } \right\} =1+o(1).
$$ 
concluding the Part 1.  

\medskip
{\bf Part 2}: By the definiton of $\calE$, since $\tau \le s=\iprod{\beta}{\beta'} \le k$, 
$$
\| X ( \beta + \beta' )\|^2 \leq \mathbb{E}_X[\| X ( \beta + \beta' )\|^2](2+\gamma)=2n(k+s)(2+\gamma)
\le 4n k (2+\gamma). 
$$
Therefore,
\begin{align}
& \Exp_{X} \left[  \exp \left\{ \frac{\| X ( \beta + \beta' )\|^2  -  \left(2 \lambda^2 -1 \right) \norm{X(\beta-\beta')}^2 }{4\sigma^2(2\lambda^2-1)}  \right\} 
\indc{ \calE  } (X, \beta) \indc{ \calE  } (X, \beta')  \right]
\indc{ \inner{\beta,\beta'} > \tau} \nonumber  \\
& \le   \Exp_{X} \left[ \exp \left\{ \frac{4nk(2+\gamma) -  \left(2 \lambda^2 -1 \right) \norm{X(\beta-\beta')}^2}{4\sigma^2(2\lambda^2-1)}  \right\}  \right]
\indc{ \inner{\beta,\beta'}> \tau} \nonumber \\
& =  \exp \left\{ \frac{nk(2+\gamma) }{\sigma^2(2\lambda^2-1)}  \right\} 
\left( 1+ \frac{  (k-s) }{\sigma^2} \right)^{-n/2}
\indc{ \inner{\beta,\beta'} > \tau},  \label{eq:condition_key_step}
\end{align}
where the first inequality follows from the definition of event $\calE$ and the last equality holds due to \prettyref{eq:part_1_moment_2}.
It follows that
\begin{align}
& \frac{\lambda^{2n}}{ (2\lambda^2-1)^{n/2} } \Exp_{\beta \independent \beta'}\left[ \Exp_{X} \left[  e^{ \frac{\| X ( \beta + \beta' )\|^2  -  \left(2 \lambda^2 -1 \right) \norm{X(\beta-\beta')}^2 }{4\sigma^2(2\lambda^2-1)} }
\indc{ \calE  } (X, \beta) \indc{ \calE  } (X, \beta')  \right]
\indc{ \inner{\beta,\beta'} > \tau} \right]  \nonumber \\
& \le \frac{\lambda^{2n}}{ (2\lambda^2-1)^{n/2} } 
 \exp \left\{ \frac{nk(2+\gamma) }{\sigma^2(2\lambda^2-1)} \right\}
\Exp_{S \sim \mathrm{Hyp}{(p,k,k)}} \left[ 
\left( 1+ \frac{  (k-S) }{\sigma^2} \right)^{-n/2} \indc{ S > \tau}  \right] \nonumber \\
& \overset{(a)}{\le} \lambda^n
e^{ n(1+\gamma/2 )}
\Exp_{S \sim \mathrm{Hyp}{(p,k,k)}} \left[ 
\left( 1+ \frac{  (k-S) }{\sigma^2} \right)^{-n/2} \indc{ S > \tau}  \right] \nonumber \\
& \overset{(b)}{=} e^{ n(1+\gamma/2 )} \Exp_{S \sim \mathrm{Hyp}{(p,k,k)}} \left[ 
\left( 1-\frac{  S }{k+\sigma^2} \right)^{-n/2} \indc{ S > \tau}  \right],
\label{compli}
\end{align}
where $(a)$ follows due to $2\lambda^2-1 \ge \lambda^2$ and $2\lambda^2-1 \ge 2 k/\sigma^2$; 
$(b)$ follows by plugging in $\lambda^2=k/\sigma^2+1$. 

Recall that $n \le (1-\alpha) n^*$. Then under our choice of $\alpha$ and $\tau$, 
applying \prettyref{lmm:MGF_large_regime} with $n$ being replaced by $n/2$, 
$c= p^{-1/2-\delta}$, 
we get that there exits a universal constant $C>0$ such that  
if $k/\sigma^2 \geq C$ then 
\begin{align*}
& e^{ n(1+\gamma/2 )} \Exp_{S \sim \mathrm{Hyp}{(p,k,k)}} \left[ 
\left( 1-\frac{  S }{k+\sigma^2} \right)^{-n/2} \indc{ S > \tau}  \right] \\
& \le \exp \left( - \alpha k \log \frac{p}{k} + \log \frac{2- c }{1-c  } + n \left( 1+ \frac{\gamma}{2}  \right) \right) \\
& \overset{(a)}{=} \exp \left( - \frac{1}{4} \alpha k \log \frac{p}{k} + \log \frac{2- c}{1-c }  \right) \\
& \overset{(b)}{\le} \exp \left( - 8 k \log \log \frac{p}{k} + \log \frac{2- c }{1-c  }  \right) =o_p(1)
\end{align*}
where $(a)$ follows because under our choice of $\gamma$ and $\alpha$,
$$
n\left( 1+ \frac{\gamma}{2}  \right) \le n + \frac{1}{2} \alpha k \log \frac{p}{k} \le n^* + \frac{1}{2} \alpha k \log \frac{p}{k} 
\le  \frac{3}{4} \alpha k \log \frac{p}{k};
$$ 
$(b)$ holds due to $\alpha k \log (p/k) \ge 32 k \log \log (p/k)$. 

Combing the bounds for Parts 1 and 2, 
we conclude $$
\chi^2(P_{\mathcal{E}} \| Q)=\Exp_{Q}\left[ \left( \frac{P_{\mathcal{E}}}{Q}\right)^2 \right]-1=o(1),
$$ 
as desired.

\end{proof}

\section{Proof of Negative Results for Recovery}
\label{sec:negative_recovery}

\label{NegR}

\subsection{Lower Bound on MSE} 
Our first result provides a connection between the relative entropy $D(P\|Q_\lambda)$ and the MSE of an estimator that depends only a subset of the observations. This bound is general in the sense that it holds for any distribution on $\beta$ with $\expect{ \|\beta\|^2}  = k$. 
For ease of notation, we write $Q_\lambda$ as $Q$ whenever the context is clear.  

\begin{lemma}\label{lem:MSE_LB}
Given an integer $n \ge 2$ and an integer $m \in \{1,\dots,n-1\}$, let $\hat{\beta}$ be an estimator that is a function of $X$ and  the first $m$ observations $(Y_1, \dots, Y_m)$. Then,  
\begin{align}
\MSE \left( \hat{\beta} \right) \ge    e^{- \frac{2}{n-m} D(P||Q)  }  (\sigma^2 + k)  - \sigma^2.
\end{align}
\end{lemma} 
\begin{proof} The conditional mutual information  $I(\beta ; Y \mid X)$ can be rewritten as
\begin{align}
I( \beta ; Y \mid X) 
& = \expects{ \log \frac{P(Y|X, \beta)}{ P(Y|X )}  }{(\beta, X, Y) \sim P} \nonumber \\
& = \expects{ \log \frac{P(Y|X, \beta)}{ Q(Y)}  }{(\beta, X, Y) \sim P} +  \expects{ \log \frac{Q(Y)}{ P(Y|X)}  }{(X, Y) \sim P}, \nonumber
\end{align}
where $(\beta, X, Y) \sim P$ denotes that $(\beta, X, Y)$ are generated according to the planted model.
Plugging in the expression of $P(Y|X, \beta)$ and $Q(Y)$, we get that 
$$
\expects{ \log \frac{P(Y|X, \beta)}{ Q(Y)}  }{(\beta, X, Y) \sim P}
=\frac{n}{2} \log(\lambda^2) + \frac{1}{2} \expect{ \frac{\|Y\|_2^2}{\lambda^2 \sigma^2} - \frac{\| Y- X \beta\|_2^2}{\sigma^2} }.
$$
Furthermore, by definition, 
$$
\expects{ \log \frac{Q(Y)}{ P(Y|X)}  }{(X, Y) \sim P}
= - D(P\|Q) 
$$
Combining the last three displayed equations gives that 
\begin{align}
I( \beta ; Y \mid X) 
& =\frac{n}{2} \log(\lambda^2) + \frac{1}{2} \expect{ \frac{\|Y\|_2^2}{\lambda^2 \sigma^2} - \frac{\| Y- X \beta\|_2^2}{\sigma^2} } - D(P\|Q)  \notag\\
& =\frac{n}{2} \left[\log\left(\frac{\lambda^2}{ 1+ k/\sigma^2}\right) +    \frac{ 1 + k/\sigma^2}{\lambda^2 } - 1 \right] + \frac{n}{2} \log(1   +k / \sigma^2)  - D(P\|Q)  \notag\\
& \ge \frac{n}{2} \log(1   +k / \sigma^2)   - D(P \|Q ). \label{eq:MSE_LB_a}
\end{align}
where the inequality follows from the fact that $\log(u)  + 1/u - 1 \ge 0$ for all $u >0$.

To proceed, we will now provide an upper bound on $I(\beta;  Y \mid X)$ in terms of the MSE. Starting with the chain rule for mutual information, we have
\begin{align}
I(\beta;  Y \mid X) & = I(\beta;  Y_1^m \mid X )  + I(\beta; Y_{m+1}^n \mid X, Y_1^m), \label{eq:MSE_LB_b}
\end{align}
where we have used the shorthand notation $Y_i^j = (Y_i, \dots, Y_j)$. Next, we use the fact that mutual information in the Gaussian channel under a second moment constraint is maximized by the Gaussian input distribution. Hence, 
\begin{align}
I(\beta;  Y_1^m \mid X ) & \le  \sum_{i=1}^m I( \beta; Y_i \mid X) \notag \\
& \le  \frac{m}{2} \expect{ \log \left(   \expect{ \|Y_1 \|^2 \mid X} /\sigma^2 \right)} \notag \\
& \le \frac{m}{2} \log \left(   \expect{ \|Y_1\|^2 } /\sigma^2 \right)  \notag \\
& \le \frac{m}{2} \log \left( 1+ k/\sigma^2 \right),  \label{eq:MSE_LB_d}
\end{align}
and 
\begin{align}
I(\beta; Y_{ m+1}^n \mid X, Y_1^m ) & \le \sum_{i=m+1}^n I ( \beta ; Y_i \mid X, Y_1^m ) \notag  \\
& \le \frac{n-m}{2} \log \left(  \expect{  \| Y_{m+1} - \expect{ Y_{m+1} \mid X, Y_1^m } \|^2 }  /\sigma^2 \right) \notag \\
& \le \frac{n-m}{2} \log \left( 1+ \MSE(\hat{\beta})/\sigma^2 \right),  \label{eq:MSE_LB_e}
\end{align}
where the last inequality holds due to 
\begin{align*}
\expect{ \|Y_{m+1} - \expect{ Y_{m+1} \mid X, Y_1^n} \|^2}  
& = \expect{\left \| \beta  - \expect{\beta \mid Y_1^m , X} \right\|^2} + \sigma^2 
\le \MSE(\hat{\beta}) + \sigma^2. 
\end{align*}
Plugging inequalities \eqref{eq:MSE_LB_d} and \eqref{eq:MSE_LB_e} back into  \eqref{eq:MSE_LB_b} leads to 
\begin{align}
I(\beta;  Y \mid X)  & \le \frac{m}{2} \log(1 + k/\sigma^2)  +\frac{n-m}{2} \log(1 + \MSE(\hat{\beta}) / \sigma^2 ). \label{eq:MSE_LB_f}
\end{align}
Comparing \eqref{eq:MSE_LB_f} with \eqref{eq:MSE_LB_a} and rearranging terms gives the stated result. 
\end{proof}

\subsection{Upper Bound on Relative Entropy via Conditioning}

We now show how a conditioning argument can be used to upper bound the relative entropy. 
Recall that \prettyref{eq:DPQ_convex_UB} implies
\begin{align}
D(P ||Q) \le (1- \eps)  D( P_{\calE} ||Q)  + \eps D( P_{\calE^c} ||Q) . \label{eq:KL_convex_UB}
\end{align}
The next result provides an upper bound on the second term on the right-hand side. 

\begin{lemma}\label{lmm:DEc_UB}
For any $\calE \subset \reals^p \times \reals^{n \times p}$ we have
\begin{align*}
\eps  D( P_{\calE^c} ||Q)  \le  2 \sqrt{ \eps}   + \frac{\eps n}{ 2} \log( \lambda^2)
 +\frac{\sqrt{  \eps} \,  n ( 1 + k/\sigma^2) }{  \lambda^2} ,
\end{align*}
where $\eps = \prob{(X,\beta) \in \calE^c}$. In particular, if $\lambda^2 = 1 + k/\sigma^2$, then 
\begin{align*}
\eps  D( P_{\calE^c} ||Q)  \le   \frac{\eps n}{ 2} \log( 1 + k/\sigma^2)
 +\sqrt{  \eps}  ( 2 + n).
\end{align*}
\end{lemma}
\begin{proof}
Starting with the definition of the conditioned planted model in \prettyref{eq:planted_cond}, we have
\begin{align*}
P_{\calE^c}(X, Y) 
& = \frac{ \Exp_{\beta} \left[  P( X, Y  \mid \beta) \indc{\calE^c}(X, \beta)   \right]} {\prob{ \calE^c  }} 
=   \frac{ P(X)  \Exp_{\beta} \left[P( Y  \mid  X, \beta) \indc{\calE^c}(X, \beta) \right]   } {\epsilon } 
\end{align*}
Recall that $W_{ij} \iiddistr \calN(0,\sigma^2)$. It follows that  $ P( Y \mid  \beta, X) \le (2 \pi  \sigma^2)^{-n/2}$  and thus
\begin{align*}
P_{\calE}(X, Y) & \le \frac{  P(X) \Exp_{\beta} \left[   \indc{\calE} (\beta, X)    \right]}{\eps (2 \pi  \sigma^2)^{n/2}} \le \frac{  P(X)}{\eps (2 \pi  \sigma^2)^{n/2}}.
\end{align*}
Therefore, recalling that  $Q(X,Y) = P(X) Q(Y)$, we have
\begin{align*}
D( P_{\calE^c} ||Q) 
& = \Exp_{ P_{\calE^c}} \left[ \log \frac{ P_{\calE^c}(X, Y)}{ P(X)Q(Y)} \right]\\
&  \le    \Exp_{ P_{\calE^c}} \left[ \log \frac{ 1 }{ \eps\, ( 2 \pi \sigma^2)^{n/2}  Q(Y)} \right]  \\
&  = \log\frac{1}{\eps} +  \frac{n}{2} \log(  \lambda^2) + \frac{\expect{ \| Y\|^2  \mid  (X,\beta) \in \calE^c } }{2 \lambda^2 \sigma^2}   
\end{align*}
Multiplying both sides by $\eps$ leads to 
\begin{align*}
\eps  D( P_{\calE^c} ||Q) 
& \le   \eps  \log\frac{1}{ \eps} +    \frac{\eps\, n}{2} \log(  \lambda^2)  + \frac{\Exp \left[  \| Y\|^2 \indc{\calE^c }(\beta, X ) \right] }{2 \lambda^2 \sigma^2} 
\end{align*}
The first term on the right-hand side satisfies $\eps \log (1/\eps)  \le 2 \sqrt{\eps}$. Furthermore, by the Cauchy-Schwarz inequality, 
\begin{align*}
\expect{  \| Y\|^2 \indc{\calE^c}(\beta, X ) } & \le  \sqrt{\expect{ \indc{\calE^c}(X,\beta)} \,  \expect{ \|Y\|^4}} =  \sqrt{ \eps n (2 + n) } ( k + \sigma^2),
\end{align*}
where we have used the fact that $\|Y\|^2/(k + \sigma^2)$ has a chi-squared distribution with $n$ degrees of freedom. Combining the above displays and using the inequality $n+2 \le 3 n$ leads to the stated result. 
\end{proof}

\subsection{Proof of \prettyref{thm:recovery}}
We are ready to prove \prettyref{thm:recovery}.
\begin{proof}[Proof of \prettyref{thm:recovery}]
First, we prove \prettyref{eq:KL_bounded} under the theorem assumptions.
Let $\calE$ be $\calE_{\gamma,\tau}$ with $\gamma$ and $\tau$ given in \prettyref{thm:main}.
It follows from \prettyref{thm:main} that 
$
D(P_\calE \|Q_{\lambda_0})=o_p(1).
$
Moreover, it follows from  \prettyref{lmm:conditioning} and $k=o\left(p\right)$ that
$$
\eps = \prob{ \calE^c} \le e^{- 4 k \log \log (p/k) }. 
$$
Thus we get from \prettyref{lmm:DEc_UB} that for $\lambda^2=k/\sigma^2+1$ and
\begin{align*}
 \eps  D( P_{\calE^c} ||Q_{\lambda_0})  & \le   \frac{\eps n}{ 2} \log \left( 1 + k/\sigma^2 \right)+\sqrt{ \eps} \, (2 +  n)   \\
& \le \frac{\eps n^*}{ 2} \log \left( 1 + k/\sigma^2 \right)+\sqrt{ \eps} \, (2 +  n^* )   \\
& \le  e^{- 4 k \log \log (p/k) } \left(  k \log \frac{p}{k} \right) +  2e^{- 2 k \log \log (p/k) } \left( 1 + \frac{k\log (p/k)}{\log (1+k/\sigma^2)} \right) 
=o_p(1),
\end{align*}
where the last equality holds due to $k=o(p)$ and $k/\sigma^2 \ge C$ for a sufficiently large constant $C$. 
In view of the upper bound in \prettyref{eq:KL_convex_UB}, we immediately get
$
D(P\|Q_{\lambda_0})=o_p(1)
$
as  desired. 

Next we prove \prettyref{eq:MSE_bound}. 
Note that  if $ \lfloor (1- \alpha)n^* \rfloor \le 1$, then  \prettyref{eq:MSE_bound} is trivially true. 
Hence, we assume $ \lfloor (1- \alpha)n^* \rfloor \ge 2$ in the following. 
Applying  Lemma~\ref{lem:MSE_LB} with $n= \lfloor (1- \alpha)n^* \rfloor$ and $m =\lfloor (1 -\alpha ) n^*\rfloor -1$ yields that 
\begin{align}
\frac{\MSE(\hat{\beta})}{ k} 
\ge \left( 1 + \frac{ \sigma^2}{ k } \right) \exp\left\{ - 2D(P||Q_{\lambda_0}) \right\}  - \frac{ \sigma^2}{k} = 1-o_p\left(1\right).
\end{align}
where the last equality holds because $D(P||Q_{\lambda_0}) = o_p(1)$ and $k/\sigma^2 \ge C$ for a constant $C$.
\end{proof}

\section{Proof of Positive Results for Recovery and Detection}\label{sec:positive}

In this section we state and prove the positive result.  

\subsection{Proof of Theorem \ref{thm:positive}}\label{PosR}

Towards proving  \prettyref{thm:positive},  we need the following lemma.

\begin{lemma}\label{chilem} 
Let $X \in \mathbb{R}^{n \times p}$ with i.i.d.\  $\mathcal{N}(0,1)$ entries and $W \sim N(0, \sigma^2 I_n)$.
Furthermore, assume that $\beta, \beta' \in \{0,1\}^p$ are two $k$-sparse vectors with $\|\beta  -\beta'\|^2 = 2 \ell$ for some $\ell \in \{1,\dots, k\}$. 
Then
$$\prob{ \| W + X (\beta - \beta') \|^2 \le \| W  \|^2 } \leq \left( 1 + \frac{ \ell}{ 2\sigma^2} \right)^{-n/2} .$$
\end{lemma}

\begin{proof}Let $Q(x)$ be the complementary cumulative distribution function of the standard Gaussian distribution, that is for any $x \in \mathbb{R}$, $Q(x)=\mathbb{P} \left[Z \geq x \right]$ for $Z \sim \mathcal{N}(0,1)$. The Chernoff bound gives $Q(x) \leq e^{-x^2/2}$ for all $x \ge 0$. Then 
\begin{align*}
& \prob{ \| W + X (\beta - \beta') \|^2 \le \| W  \|^2 } \\
& = \prob{ 2  W^T X(\beta - \beta')  + \| X (\beta - \beta') \|^2 \le  0  } \\
& = \prob{ \frac{  -   W^T X(\beta - \beta') }{\sigma  \| X(\beta - \beta')\| }  \ge  \frac{\| X (\beta - \beta')\|}{ 2 \sigma}    } \\
& \overset{(a)}{=} \expect{Q\left( \frac{\| X (\beta - \beta')\|}{ 2 \sigma}   \right)   }  \\
& \overset{(b)}{\le}  \expect{\exp\left( - \frac{\| X (\beta - \beta')\|^2}{ 8 \sigma^2}   \right)   } \\
& \le \left( 1 + \frac{ \ell}{ 2\sigma^2} \right)^{-n/2},
\end{align*} 
where $(a)$ holds because conditioning on $X$, 
$\frac{  -   W^T X(\beta - \beta') }{\sigma  \| X(\beta - \beta')\| }  \sim \mathcal{N}(0,1)$;
$(b)$ holds due to $Q(x) \leq e^{-x^2/2}$;
 the last inequality follows from $\| X (\beta - \beta')\|_2^2/ (2\ell) \sim  \chi^2(n)$ 
 and $\mathbb{E}_{Z \sim \chi^2(1)}\left[e^{-tZ}\right]=\frac{1}{\sqrt{1+2t}}$ 
for $t>0$.
\end{proof}

We now proceed with the proof of  \prettyref{thm:positive}.
\begin{proof}[Proof of \prettyref{thm:positive}]
First, note that when $k=o(p)$, \prettyref{eq:strong_recovery_positive} readily follows from \prettyref{eq:recovery_positive}. 
In particular, observe that since $\hat{\beta},\beta \in \{0,1\}^p$ are binary $k$-sparse vectors, 
it follows that $\|\hat{\beta}-\beta\|^2 \leq 2k$ and therefore
\begin{align*}
\frac{1}{k} \MSE \left( \hat{\beta}\right)&= \frac{1}{k} \expect{ \| \hat{\beta}-\beta\|^2}  \\
&\leq \frac{2}{\log \left(p/k\right)}+2\mathbb{P}\left[\|\hat{\beta}-\beta\|^2 \geq \frac{2k}{\log \left(p/k \right)} \right]\\
& \leq \frac{2}{\log \left(p/k\right)}+\frac{2e^2}{\log^2 \left(p/k\right) \left(1-e^{-1}\right)},
\end{align*}
which is $o_p\left(1\right)$ when $k=o\left(p\right).$

It remains to prove \prettyref{eq:recovery_positive}. Set for convenience
\begin{align}
d \triangleq  \left\lceil  \frac{k}{\log (p/k)} \right\rceil. \label{eq:def_d}
\end{align}
By the definition of the MLE, 
$$
\| W + X \left(\beta -  \hat{\beta} \right) \|^2=\|Y-X \hat{\beta}  \|^2
\leq \|Y-X\beta\|^2=\|W\|^2.
$$ 
Hence,
$$
\left\{ \|\hat{ \beta} - \beta\|^2 \ge 2d \right\} = 
\cup_{\ell=d}^{k} \left\{ \exists \beta' \in \{0,1\}^p: \| \beta'\|_0=k, 
\| \beta' - \beta\|^2 =2 \ell, \| W + X (\beta - \beta') \|^2 \le \| W  \|^2
\right\}. 
$$
By a union bound and Lemma \ref{chilem}, we have that
\begin{align}\label{UB}
\prob{ \| \hat{ \beta} - \beta\|^2 \ge 2d } &  \le \sum_{\ell =d }^k \binom{k}{\ell}  \binom{p-k}{\ell} \left( 1 + \frac{ \ell}{ 2 \sigma^2} \right)^{-n/2} \notag \\
& 
\overset{(a)}{\le}  \sum_{\ell =d }^k \left( \frac{ke}{\ell} \right)^\ell \left( \frac{pe}{\ell} \right)^\ell \left( 1 + \frac{ \ell}{ 2 \sigma^2} \right)^{-n/2} \notag  \\
& \overset{(b)}{\le} \sum_{\ell =d }^k \left( \frac{e^2 pk}{d^2} \right)^\ell \left( 1 + \frac{ \ell}{ 2 \sigma^2} \right)^{-n/2}   \notag  \\
& \triangleq  \sum_{\ell =d }^k \exp \left( h(\ell) -\ell \right), 
\end{align}
where $(a)$ holds due to  $\binom{m_1}{m_2} \leq (em_1/m_2)^{m_2}$; $(b)$ holds due to $ \ell \geq d$; and
$$
h(x) \triangleq   - \frac{n}{2} \log \left( 1 + \frac{ x}{ 2 \sigma^2} \right)  + x \log \left( \frac{e^3 pk}{d^2} \right) .
$$
Note that $h(x)$ is convex in $x$; hence the maximum of $h(\ell)$
for $\ell \in [d, k]$ is achieved at either $\ell=d$ or $\ell=k$, \ie,
\begin{align}
\max_{d \le \ell \le k}h(\ell) \le \max \left\{ h(d), h(k) \right\}. \label{eq:max_h_ell}
\end{align}
We proceed to upper bound $h(d)$ and $h(k)$. 
Note that 
\begin{align} \label{eq:F1}
\left( 1+ \frac{\log 2}{ \log \left( 1+  k/(2\sigma^2) \right)} \right)
\log \left( 1 + \frac{ k}{ 2 \sigma^2} \right) \ge \log \left(1+ k/\sigma^2 \right).
\end{align}
Thus, it follows from \prettyref{eq:recovery_cond} that
\begin{align}
n \ge  \frac{  \log \left(1+ k/\sigma^2 \right)}{\log \left( 1 + \frac{ k}{ 2 \sigma^2} \right) }  \left( 1 + \frac{4 \log \log (p/k)}{\log (p/k)} \right) n^*
= \frac{2k \log (p/k)}{ \log \left( 1 + \frac{ k}{ 2 \sigma^2} \right)}  \left( 1 + \frac{4 \log \log (p/k)}{\log (p/k)} \right). \label{eq:F2}
\end{align}
Then we conclude that 
\begin{align} \label{F3}
h(k)&=- \frac{n}{2} \log \left( 1 + \frac{ k}{ 2 \sigma^2} \right)  + k \log \left( \frac{e^3 pk}{d^2} \right) \notag \\
& \overset{\prettyref{eq:F2}}{\le} - k \log (p/k) -4k \log \log (p/k) +   k \log \left( \frac{e^3 pk}{d^2} \right) \notag \\
& \overset{\prettyref{eq:def_d}}{\le}    - k \log (p/k) -4k \log \log (p/k) 
+k \log \left( \frac{e^3 pk \log^2 (p/k)}{k^2} \right) \notag \\
&= - 2k \log \log (p/k) + 3k.
\end{align}

Analogously, we can upper bound $h(d)$ as follows:
\begin{align}\label{Hd}
h(d) &=- \frac{n}{2} \log \left( 1 + \frac{ d}{ 2 \sigma^2} \right)  
+ d \log \left( \frac{e^3 pk}{d^2} \right)  \notag \\
& \overset{\prettyref{eq:F2}}{\le} - \left( 1+ \frac{4 \log \log (p/k)}{\log (p/k)} \right)  \frac{k \log (p/k) }{ \log \left(1+ k/(2\sigma^2) \right)}
\log \left( 1 + \frac{ d}{ 2 \sigma^2} \right) + d \log \left( \frac{e^3 pk}{d^2} \right).
\end{align}
Let 
$$
q(x) \triangleq \log \left( 1 + \frac{ x}{ 2 \sigma^2} \right) - \frac{x}{k} \log \left(1+ \frac{k}{2\sigma^2} \right) 
$$
Note that $q(x)$ is concave in $x$, $q(0)=0$, and $q(k)=0$. Thus 
$$\min_{ 0 \le x \le k} q(x) \ge \min \left\{ q(0), q(k) \right\} \ge 0.$$
Hence, $q(d) \ge 0 $, \ie,
$$
 k \log \left( 1 + \frac{ d }{ 2 \sigma^2} \right)  \ge d \log \left(1+ \frac{k}{2\sigma^2} \right) .
$$
Combining the last displayed equation with (\ref{Hd}) gives that 
\begin{align*}
h(d) & \le - \left( 1+ \frac{4 \log \log (p/k)}{\log (p/k)} \right)  d \log (p/k)  + d \log \left( \frac{e^3 pk}{d^2} \right)\\
& \overset{\prettyref{eq:def_d}}{\le} - d \log (p/k) - 4d \log \log (p/k)  
 + d \log \left( \frac{e^3 pk \log^2 (p/k)}{k^2} \right) \\
 & \le   - 2d \log \log (p/k) +  3d. 
\end{align*}
Combining the last displayed equation with (\ref{F3}) and \prettyref{eq:max_h_ell}, we get that 
$$
\max_{d \le \ell \le k}h(\ell)  \leq - 2d \log \log (p/k) + 3d.
$$
Combining the last displayed equation with (\ref{UB}) yields that 
\begin{align*}
\prob{ \| \hat{ \beta} - \beta\|^2 \ge 2d }
& \le e^{-2d \log \log (p/k)+3d} \sum_{\ell=d}^k e^{-\ell}  \\
& \le e^{-2 d \log \log (p/k)+3d} \frac{e^{-d}}{1-e^{-1}}\\
& \le e^{-2\log \log (p/k)} \frac{e^2}{1-e^{-1}} \\
&= \frac{e^2}{(1-e^{-1}) \log^2 (p/k)},
\end{align*} 
where the last inequality holds under the assumption $\log \log (p/k) \ge 1$. 
This completes the proof of \prettyref{thm:positive}.

\end{proof}

\subsection{Proof of Theorem \ref{thm:posDetection}} \label{posdet}

\begin{proof}
Under the planted model, we have
\[
\calT(X,Y) \le \frac{\| W\|^2}{ \| W + X \beta\|^2}. 
\]
Note that $\|W\|^2/\sigma^2 \sim \chi^2 (n)$ and 
$\| W + X \beta\|^2/ (k+\sigma^2) \sim \chi^2(n)$.  
It follows from the concentration inequality for chi-square distributions that
$$
\prob{ \|W\|^2 \ge \sigma^2 \left( n + 2\sqrt{nt} + 2 t \right) } \le e^{-t},
$$
and 
$$
\prob{ \|W + X \beta \|^2 \le(k+ \sigma^2) \left( n - 2\sqrt{nt}  \right) } \le e^{-t}.
$$
Therefore, for any $t_n$ such that $t_n \to +\infty$ as $n \to +\infty$,
$$
P\left(\calT(X,Y) \ge \frac{\sigma^2}{ k + \sigma^2}  \frac{ n + 2\sqrt{nt_n} + 2t_n}{n - 2 \sqrt{n t_n} } \right) \to 0. 
$$In particular, using for example $t_n=\log n=o\left(n\right)$ we have  $\frac{ n + 2\sqrt{nt_n} + 2t_n}{n - 2 \sqrt{n t_n} }=1+o\left(1\right)$, we can easily conclude from the definition of $\tau$ that
$$
P\left(\calT(X,Y) \ge \tau \right) \to 0. 
$$
Meanwhile, under the the null model, we have 
\begin{align*}
\calT(X,Y) = \frac{\min_{\beta' \in \{0,1\}^p, \| \beta'\|_0 =k } \| \lambda W - X \beta\|^2}{ \| \lambda W \|^2}.
\end{align*}
Note that  $W$ and $X$ are independent; thus we condition on $X$ in the sequel. We have
\begin{align}
& \expect{ \min_{\beta' \in \{0,1\}^p, \| \beta'\|_0 =k } \| \lambda W - X \beta\|^2  } \nonumber \\
&\ge \min_{z_1, \dots, z_M \in \reals^n}  \expect{ \min_{m \in [M] } \| \lambda W - z_m\|^2} \nonumber \\
&\ge  \expect{\| \lambda W\|^2} M^{-2/n} \nonumber \\
& =n \lambda^2 \sigma^2 M^{-2/n},  \label{eq:mean_squared_bound}
\end{align}
where $M = \binom{p}{k}$ and the last inequality holds because the distortion rate function $D(R) = \sigma^2 \exp( 2 R) $ provides a non-asymptotic lower bound on the distortion of an i.i.d.\ $\calN(0, \sigma^2)$ source with rate $R = \frac{1}{n} \log M$ 
(See \eg~\cite[Section 10.3.2]{cover2006elements}).

Define $f: \reals^n \to \reals$, 
$$
f(w) =  \min_{\beta' \in \{0,1\}^p, \| \beta'\|_0 =k } \|  \lambda w - X \beta\| 
$$
It follows that $f$ is $\lambda$-Lipschitz and thus in view of the Gaussian concentration inequality for Lipschitz functions (see, \eg~\cite[Theorem 5.6]{boucheron2013concentration}),
we get that 
\begin{align}
\prob{ \left| f(W) -  \expect{f(W)   }  \right | \ge t   } \le 2 \exp \left( - \frac{t^2}{2 \lambda^2\sigma^2}  \right).
\label{eq:gaussian_lips_concent}
\end{align}
Thus
\begin{align*}
\var \left( f(W) \right) & = \expect{ \left( f (W)  - \expect{f(W)   } \right)^2 } \\
&= \int_{0}^\infty  \prob{ \left(  f(W) -  \expect{f(W)  }  \right)^2  \ge t  } d t  \\
& \le \int_{0}^\infty  2 \exp \left( - \frac{t}{2 \lambda^2\sigma^2}  \right) d t  \\
& =4 \lambda^2 \sigma^2. 
\end{align*}
Combining the last displayed equation with \prettyref{eq:mean_squared_bound} gives that 
$$
\expect{ f(W) } \ge \sqrt{ \expect{ f^2(W) } - 4 \lambda^2 \sigma^2 } \ge \lambda \sigma \sqrt{ n M^{-2/n} -4 }.
$$
Combining the last displayed equation with \prettyref{eq:gaussian_lips_concent}, we get that for any
$t_n$ such that $t_n \to +\infty$ as $n \to +\infty$, 
$$
\prob{ f(W) \le \lambda \sigma \sqrt{ n M^{-2/n} -4 } - \lambda \sigma t_n }  \to 0.
$$
Also, it follows from the concentration inequality for chi-square distributions that
$$
\prob{ \|W  \|^2 \ge   \sigma^2 \left( n + 2\sqrt{nt_n}  + 2t_n \right) } \to 0. 
$$
Thus, recalling that $T(X,Y)= f^2(W)/ \| \lambda W\|^2$, we get that
\begin{equation}\label{Qbound}
Q \left( T(X, Y)  \le \frac{  \left( \sqrt{ n M^{-2/n} -4 } - t_n \right)^2  }{ \left( n + 2\sqrt{nt_n}  + 2t_n \right) } \right) \to 0. 
\end{equation}

By assumption \prettyref{eq:ass_detection_cond2}, there exists a positive constant $\alpha>0$ such that
$$
n \ge  \frac{2 \log M}{ \log \left(1+k/\sigma^2 \right) + \log (1- \alpha) }.
$$
It follows that 
$$
M^{2/n} \le  (1-\alpha)\left( 1+ k/\sigma^2 \right)  .
$$
Since 
$$
\tau =\frac{1}{(1-\alpha/2)\left( 1+ k/\sigma^2 \right) }
$$
we have
$$
 \tau < \frac{1}{(1-\alpha)\left( 1+ k/\sigma^2 \right) } \leq M^{-2/n}.
$$
By assumption \prettyref{eq:ass_detection_cond}, $n M^{-2/n} \to +\infty$. Hence,
there exists a sequence of $t_n$ such that 
$t_n \to +\infty$ and $t_n=o(\sqrt{n} M^{-1/n}).$ 
In particular, for this choice of $t_n$, combining the above we have 
$$\liminf_n \frac{  \left( \sqrt{ n M^{-2/n} -4 } - t_n \right)^2  }{ \left( n + 2\sqrt{nt_n}  + 2t_n \right) }>\tau.$$Hence from (\ref{Qbound}) we can conclude 
$$Q \left( T(X, Y)  \le \tau \right) \to 0. $$

Hence indeed,
$$
P\left( T(X, Y) \ge \tau
 \right)
 + Q \left( T(X, Y) \le \tau \right) \to 0,
$$
which shows that $T(X,Y)$ with threshold $\tau$ indeed achieves the strong detection. 

\end{proof}

\section{Conclusion and Future Work}\label{sec:open}

In this paper, we establish an \textit{All-or-Nothing} information-theoretic phase transition 
for recovering a $k$-sparse vector $\beta \in \{0,1\}^p$ from $n$ independent  linear 
Gaussian measurements  $Y=X \beta +W$ 
with noise variance $\sigma^2$. In particular, we show that 
the MMSE normalized
by the trivial MSE jumps from $1$ to $0$ at a critical sample size  $n^*=\frac{2k \log \left(p/k\right) }{\log \left(1+k/\sigma^2\right)}$ within a small window of size $\epsilon n^*$. The constant $\epsilon>0$ can be made arbitrarily small by increasing the signal-to-noise ratio $k/\sigma^2$.
Interestingly, the phase transition threshold $n^*$ is asymptotically equal to the ratio of entropy $H(\beta)$ and the AWGN channel capacity $\frac{1}{2} \log \left(1+k/\sigma^2\right)$. 
Towards establishing this All-or-Northing phase transition, we also study 
a closely related hypothesis testing problem, where the goal is to distinguish this planted model $P$ from a null model $Q_{\lambda}$ where $\left(X,Y\right)$ are independently generated and $Y_i \iiddistr \calN\left(0,\lambda^2 \sigma^2\right)$. When $\lambda = \lambda_0=\sqrt{k/\sigma^2+1}$, we show that the sum of Type-I and Type-II testing errors also jumps from $1$ to $0$ at $n^*$ within a small window of size $\epsilon n^*$.  

Our impossibility results for $n \le (1-\epsilon)n^*$ apply under a crucial assumption that 
$k \le p^{1/2-\delta}$ for some arbitrarily small but fixed constant $\delta>0$.
This naturally implies for $\Omega\left(p^{1/2}\right)\le k \le o\left(p\right)$, two open problems for the identification of the detection and the recovery thresholds, respectively. 

For detection, as argued in
\prettyref{app:assumption_needed}, $k=o\left(p^{1/2}\right)$ is needed for $n^*$ being the detection threshold, because weak detection is achieved for all $n=\Omega\left(n^*\right)$ when $k=\Omega(p^{1/2})$, that is the weak detection threshold becomes $o\left(n^*\right)$. The identification of the precise  detection threshold when $\Omega(p^{1/2}) \le k \le o\left(p\right)$ is an interesting open problem. 

For recovery, however, we believe that the recovery threshold  still equals $n^*$ when $\Omega \left(p^{1/2} \right) \le k \le o(p)$. To prove this, we propose to study the detection problem where both the (conditional) mean and the covariance are matched between the planted and null models. Specifically, let us consider a slightly modified
null model $Q$ with the matched conditional mean 
$\Expect_Q \left[ Y |X \right] = \Expect_P \left[ Y |X \right]  = \frac{k}{p} X \mathbf{1}$ and the matched covariance 
$\Expect_Q \left[YY^\top \right] = \Expect_P \left[ YY^\top \right] $, where $\mathbf{1}$ denotes the all-one vector. 
For example, if $X,W$ are defined as before and $Y\triangleq  \frac{k}{p} X \mathbf{1} +\lambda W$ 
with $\lambda$ equal to $\sqrt{\frac{k}{\sigma^2} + 1-\frac{k^2}{p}}$, then both the mean and covariance constraints are satisfied.
It is an open problem whether this new null model is indistinguishable from the planted model $P$ when $n \le \left(1-\epsilon\right)n^*$ and $\Omega \left(p^{1/2} \right) \le k \le o(p)$. If the answer is affirmative, then 
we may follow the analysis road map in this paper to further establish the impossibility of recovery. 

Finally, another interesting question for future work is to understand the extent to which the All-or-Nothing phenomenon applies beyond the binary vectors setting
or the Gaussian assumptions on $(X,W)$.  In this direction, some recent work~\cite{reeves:2017c} has shown that under mild conditions on the distribution of $\beta$, the distance between the planted and null models can be bounded in term of ``exponential moments'' similar to the ones studied in \prettyref{app:exponential_moment}.

\section*{Acknowledgment} 
G.~Reeves is supported by the NSF Grants CCF-1718494 and CCF-1750362. J.~Xu is supported by the NSF Grants CCF-1755960 and IIS-1838124.

\newpage 
\begin{appendices}

\section{Hypergeometric distribution and exponential moment bound }
\label{app:exponential_moment}

Throughout this subsection, we fix 
\begin{equation}\label{DfnTau}
\lambda^2 = k/\sigma^2 +1, \quad \text{ and } \quad
\tau=k\left(1-\frac{1}{\log^2 \lambda^2}\right).
\end{equation}
The main focus of this subsection is to give tight characterization of the following ``exponential'' moment:
$$
\Exp_{S \sim \mathrm{Hyp}(p,k,k)} \left[ \left( 1- \frac{S}{k+\sigma^2} \right)^{-n} \indc{S \in [a,b]} \right].
$$
for a given interval $[a,b]$. It turns out this ``exponential'' moment exhibit quantitatively different
behavior in the following three different regimes of overlap $S$: small regime ($s \le \epsilon k$), intermediate regime
($ \epsilon k <  s \le \tau$), and large regime ($s \ge \tau$), where $\epsilon$ is given in \prettyref{eq:epsilon_overlap_def}.

In the sequel, we first prove \prettyref{lmm:MGF_bound}, which focuses on the small and intermediate regimes under the
assumption $n \le n^*$. Then we prove \prettyref{lmm:MGF_large_regime}, which focuses on the large regime under the
assumption $n \le (1-\alpha) n^*/2$ for $\alpha \in (0,1/2)$.

We start with a simple lemma, bounding the probability mass of an hypergeometric distribution. 
\begin{lemma}\label{lmm:Hyp_bound}
Let $p,k \in \mathbb{N}$.
Then for $S \sim \mathrm{Hyp}(p,k,k)$ and any $s \in [k]$,
$$
\mathbb{P}\left( S=s\right) \leq \binom{k}{s}\left(\frac{k}{p-k+1}\right)^s.
$$ 
\end{lemma}

\begin{proof}
 We have 
\begin{align*}
\mathbb{P}\left( S=s \right)
&= \binom{k}{s} \frac{\binom{p-k}{k-s}}{\binom{p}{k}} \leq \binom{k}{s}  \frac{\binom{p}{k-s}}{\binom{p}{k}} = \binom{k}{s} \frac{(p-k)! (k)! }{(p-k+s)!(k-s)!} \leq \binom{k}{s} \left(\frac{k}{p-k+1} \right)^s.
\end{align*}
\end{proof}

Next, we upper bound the ``exponential'' moment in the small overlap regime ($s \le \epsilon k$), and the intermediate overlap regime ($ \epsilon k <  s \le \tau$).  
\begin{lemma}\label{lmm:MGF_bound}
Suppose $n \leq n^*$. 
\begin{itemize}
\item If $ k \le p^{\frac{1}{2}-\delta} $ for an arbitrarily small but fixed constant $\delta \in (0,\frac{1}{2})$ and $k/\sigma^2 \ge C(\delta)$ for a sufficiently large constant $C(\delta)$ only depending on $\delta$, then for any $0 \le \epsilon \le 1/2$, 
\begin{align}
\Exp_{S \sim \mathrm{Hyp}(p,k,k)} \left[ \left( 1- \frac{S}{k+\sigma^2} \right)^{-n} 
\indc{S \le \epsilon k}  \right ]
= 1+o_p(1), \label{eq:MGF_hyp_bound_2}
\end{align}
\item If  $k=o(p)$ and $k/\sigma^2 \ge C$ for a sufficiently large universal constant $C$, then for 
\begin{align}
\epsilon=\epsilon_{k,p}
=\frac{ \log \log (p/k) }{ 2 \log (p/k)}, \label{eq:epsilon_overlap_def}
\end{align}
it holds that 
\begin{align}
\Exp_{S \sim \mathrm{Hyp}(p,k,k)} \left[ \left( 1- \frac{S}{k+\sigma^2} \right)^{-n} 
\indc{ \epsilon k < S \le \tau }  \right ]
= o_p(1), \label{eq:MGF_hyp_bound_3}
\end{align}
\end{itemize}
\end{lemma}

\begin{proof}

Using Lemma \ref{lmm:Hyp_bound},
$$
\Exp_{S \sim \mathrm{Hyp}(p,k,k)} \left[ \left( 1- \frac{S}{k+\sigma^2} \right)^{-n}  \indc{S \le \tau} \right ]
= \prob{S=0} + \sum_{s=1}^{\floor{\tau}} \binom{k}{s} \left( \frac{k}{p-k+1} \right)^s e^{- n \log \left( 1 - \frac{s}{k+\sigma^2} \right) }.
$$
Note that 
$$
\prob{S=0}=\frac{\binom{p-k}{k}}{\binom{p}{k}} \ge \left( 1 - \frac{k}{p} \right)^k \ge 1- k^2/p =1+o_p(1),
$$
where the last equality holds due to $k \le p^{1/2-\delta}$ for some constant $\delta \in (0,1/2)$. 
Thus, to show \prettyref{eq:MGF_hyp_bound_2} it suffices to show
$$
 \sum_{s=1}^{\floor{\epsilon k}} \binom{k}{s} \left( \frac{k}{p-k+1} \right)^s e^{- n^* \log \left( 1 - \frac{s}{k+\sigma^2} \right) } = o_p(1),
$$
 and to show \prettyref{eq:MGF_hyp_bound_3} it suffices to show
$$
 \sum_{s=\ceil{\epsilon k}}^{\floor{\tau}} \binom{k}{s} \left( \frac{k}{p-k+1} \right)^s e^{- n^* \log \left( 1 - \frac{s}{k+\sigma^2} \right) } = o_p(1),
$$

We first prove \prettyref{eq:MGF_hyp_bound_2}.

\paragraph*{Proof of \prettyref{eq:MGF_hyp_bound_2}:}

Using the fact that $\binom{k}{s} \le k^s$, we have
\begin{align*}
 \sum_{s=1}^{\floor{\epsilon k}} \binom{k}{s} \left( \frac{k}{p-k+1} \right)^s e^{- n^* \log \left( 1 - \frac{s}{k+\sigma^2} \right)  }
 &\le   \sum_{s=1}^{\floor{\epsilon k}} k^s \left( \frac{k}{p-k+1} \right)^s e^{- n^* \log \left( 1 - \frac{s}{k+\sigma^2} \right) }\\
&=\sum_{s=1}^{\floor{\epsilon k}} e^{ -s \log \frac{p-k+1}{k^2} - n^* \log \left( 1 - \frac{s}{k+\sigma^2} \right) }\\
&=\sum_{s=1}^{\floor{\epsilon k}} e^{f(s) - s \log \frac{p-k+1}{p } }, 
\end{align*}
where for $s \in [1,\epsilon k]$ let the real-valued function $f$ be given by
$$
f(s) = -s \log \frac{p }{k^2 } - n^* \log \left( 1 - \frac{s}{k+\sigma^2} \right).
$$

\begin{claim}
Suppose $k \le p^{1/2-\delta}$ for a constant $\delta \in (0,1/2)$ and $\epsilon \le 1/2$. 
There exists a constant $C_1=C_1(\delta)>0$, such that if $k/\sigma^2 \geq C_1$ then 
it holds that for any $s \in [1,\epsilon k]$,  $f(s) \le -\frac{1}{2} s \log \frac{p}{k^2}$. 
\end{claim}

\begin{proof}[Proof of the Claim]
Standard calculus implies that for $x \in (0,1)$, $\log (1-x) \ge -(1+x) x$.  
Hence, for $0 \le x \le \epsilon \le 1/2,$
\begin{align}
\log(1-x) \ge - (1+\epsilon)x. \label{eq:log_bound}
\end{align}
Using this inequality it  follows that for since for any $s \in [1,\epsilon k]$ $\frac{s}{k+ \sigma^2} \leq \epsilon$, it also holds
$$
 f(s) \le -s\log \frac{p}{k^2} + n^* (1+\epsilon) \frac{s}{k+\sigma^2} 
 = s \left( - \log \frac{p}{k^2}+ \frac{n (1+\epsilon)}{k+\sigma^2}  
\right)
\le - \frac{1}{2} s \log \frac{p}{k^2},
$$
where the last inequality holds under the assumption that 
$$
n^* \le \frac{ (k+\sigma^2) \log \frac{p}{k^2} } {2 (1+\epsilon) }.
$$
Recall that $n^*=\frac{2k \log (p/k)}{\log (1+ k/\sigma^2)} $.
Hence it suffices to show that
$$
 \frac{2k \log (p/k)}{\log (1+ k/\sigma^2)} \le  \frac{ (k+\sigma^2) \log \frac{p}{k^2} } { 2 (1+\epsilon) } 
$$
which holds if and only if 
\begin{align}\label{targapp}
\left[ 1 - \frac{4(1+\epsilon)}{ (1+ \sigma^2/k) \log (1+k/\sigma^2)}\right] \log \frac{p}{k} \ge \log k.
\end{align} 
By assumption, $k \le p^{1/2-\delta}$ for  $ \delta \in (0,\frac{1}{2})$. Hence, (\ref{targapp}) is satisfied if 
$$
1 - \frac{4(1+\epsilon)}{ (1+ \sigma^2/k) \log (1+k/\sigma^2)} \geq \frac{\frac{1}{2}-\delta}{\frac{1}{2}+\delta}.$$ 
Since $\epsilon \le 1/2$, 
there exists a constant $C_1=C_1(\delta)>0$ depending only on $\delta$ such that if $ \frac{k}{\sigma^2} \geq C_1$ then the last displayed
equation is satisfied. 
This completes the proof of the claim. 
\end{proof}

Using the above claim we conclude that 
$$
\sum_{s=1}^{\floor{\epsilon k}} e^{f(s) -s \log \frac{p-k+1}{p} } \le \sum_{s=1}^{\floor{\epsilon k} } e^{ -\frac{1}{2} s \left( \log (p/k^2) + 2 \log \frac{p-k+1}{p} \right) }
\le \frac{ e^{-\frac{1}{2} \log \frac{(p-k+1)^2}{pk^2}} }{1 - e^{-\frac{1}{2} \log \frac{(p-k+1)^2}{pk^2}} }=o_p(1),
$$
where the last equality holds due to $k \le p^{\frac{1}{2}-\delta}$.

\medskip

Next we prove \prettyref{eq:MGF_hyp_bound_3}. Again it suffices to 
prove \prettyref{eq:MGF_hyp_bound_3} for $n=n^*$.

\paragraph*{Proof of \prettyref{eq:MGF_hyp_bound_3}:}
Note that 
$
\binom{k}{s} \le  2^k.
$
Hence,
\begin{align*}
 \sum_{s=\ceil{\epsilon k}}^{\floor{\tau} } \binom{k}{s} \left( \frac{k}{p-k+1} \right)^s e^{- n^* \log \left( 1 - \frac{s}{k+\sigma^2} \right)  } & \le 2^k \sum_{s=\ceil{\epsilon k}}^{ \floor{\tau}} \left( \frac{k}{p-k+1} \right)^s e^{ - n^* \log \left( 1 - \frac{s}{k+\sigma^2} \right) }\\
 & = 2^k \sum_{s=\ceil{\epsilon k}}^{ \floor{\tau}} e^{ -s \log \frac{p}{k} - n^* \log \left( 1 - \frac{s}{k+\sigma^2} \right) - s \log \frac{(p-k+1)}{p} }.
\end{align*}
Define for $s \in [0,k]$, the function $g$ given by
\begin{align}
g(s) \triangleq -s \log \frac{p}{k} - n^* \log \left( 1 - \frac{s}{k+\sigma^2} \right). 
\label{eq:g_function}
\end{align}
The function $g$ is convex in $s$ for $\epsilon k \le s \le \tau$, as the addition of two convex functions. Hence, the maximum of $g(s)$ 
over $s \in [\epsilon k, \tau]$ 
is achieved at either $s=\epsilon k$ or 
$s= \tau.$  Thus it suffices to upper bound $g(\epsilon k)$ and $g(\tau)$.

\begin{claim}\label{claim:g_bound}
 There exist a universal constant $C_2>0$ such that if $k/\sigma^2 \geq C_2$, then 
 $g(\tau) \le -\frac{1}{2} k \log (p/k)$ and
 $g(\epsilon k) \le -\frac{\epsilon k}{2} \log  \frac{p }{k} $. 
 \end{claim}

\begin{proof}[Proof of the Claim]
We first upper bound $g(\tau)$.
\begin{align*}
g\left(\tau \right) & \le -\tau \log \frac{p}{k} 
- n^* \log \left(1 - \frac{ \tau}{k} \right) \\
&=- \left(1-\frac{1}{\log^2 \lambda^2}\right)  k \log \frac{p}{k} +  \frac{4 k \log (p/k) \log \log (\lambda^2) }{ \log (\lambda^2)},
\end{align*} 
where the last equality holds by plugging in 
the expressions of $\tau$ and $n^*$.

Recall that $\lambda^2 = 1+ k/\sigma^2$.
Hence, there exists a universal constant $C_2>0$ such that if $k/\sigma^2 \ge C_2$, then
$$
- \left(1-\frac{1}{\log^2 \lambda^2}\right)  k \log \frac{p}{k} +  \frac{4 k \log (p/k) \log \log (\lambda^2) }{ \log (\lambda^2)}
\le -\frac{1}{2} k \log \frac{p}{k}.
$$ 
Combining the last two displayed equations yields that $g(\tau) \le -\frac{1}{2} k \log (p/k)$.

For $g(\epsilon k)$, applying \prettyref{eq:log_bound}, we get that
$$
g(\epsilon k) = - \epsilon k \log \frac{p}{k} - n^* \log \left( 1- \frac{\epsilon k}{k+\sigma^2} \right)
\le - \epsilon k \log \frac{p}{k} + \frac{n^* \epsilon k }{k+\sigma^2} (1+\epsilon)
=\epsilon k \left(  - \log \frac{p}{k} + \frac{n^*(1+\epsilon)} {k+\sigma^2}  \right).
$$
Note that we can conclude $g(\epsilon k) \leq -\frac{\epsilon k}{2} \log  \frac{p}{k} $ if 
$$
- \log \frac{p}{k} + \frac{n^*(1+\epsilon)} {k+\sigma^2}
\le - \frac{1}{2} \log \frac{p}{k}  
$$
which holds if and only if 

$$
 n^* = \frac{2k \log (p/k)}{\log (1+ k/\sigma^2)} \le  \frac{ (k+\sigma^2) \log (p/k) } { 2 (1+\epsilon) } 
$$
or equivalently
\begin{align*}
\frac{4(1+\epsilon)}{ (1+ \sigma^2/k) \log (1+k/\sigma^2)} \le 1.
\end{align*}
Note that there exists a universal  constant $C_2>0$ such that if  $k/\sigma^2 \ge C_2$
then the last displayed inequality is satisfied 
and hence 
$
g(\epsilon k) 
\leq -\frac{\epsilon k}{2} \log  \frac{p }{k} 
$
where the last inequality holds by choosing $C_2$ sufficiently large. 
\end{proof}

Using the above claim we now have that if $k/\sigma^2 \geq C_2$, 
\begin{align*}
 \sum_{s=\ceil{\epsilon k} }^{\floor{\tau}} \binom{k}{s} \left( \frac{k}{p-k+1} \right)^s e^{- n^* \log \left( 1 - \frac{s}{k+\sigma^2} \right)  } &
 \le 2^k \sum_{s=\ceil{\epsilon k}}^{ \floor{\tau } } e^{ g(s) - s \log \frac{(p-k+1)}{p} }\\
& \le  e^{ k\log 2 + \log k- \frac{\epsilon k}{2}  \log \frac{p}{k}- k  \log \frac{(p-k+1)}{p} } =o_p(1),
\end{align*}
where the last equality holds due to $\log k \le k$, $k=o(p)$, and that
$$
\frac{\epsilon k}{2} \log  \frac{p}{k} 
= - \frac{k}{4} \frac{\log \log (p/k)}{\log (p/k)}   \log  \frac{p}{k} 
=- \frac{k}{4} \log \log (p/k).
$$

\end{proof}

Finally, we upper bound the ``exponential'' moment in the large overlap regime ($s \ge \tau$) where $\tau$ is defined in (\ref{DfnTau}).
\begin{lemma}\label{lmm:MGF_large_regime}
Suppose that $k \le c p$ for $c \in (0,1)$ and $k/\sigma^2 \ge C$ for a sufficiently large universal constant $C$. 
If $n \leq \frac{1}{2} (1-\alpha)n^*$ for some $\alpha \le 1/2$,
then
\begin{align}
\Exp_{S \sim \mathrm{Hyp}(p,k,k)} \left[ \left( 1- \frac{S}{k+\sigma^2} \right)^{-n}  \indc{ S \ge \tau }  \right ] 
\le \exp \left( -\alpha k \log \frac{p}{k} + \log \frac{2-c}{1-c} \right). \label{eq:MGF_hyp_bound_1}
\end{align}
\end{lemma}

\begin{proof}

Using Lemma \ref{lmm:Hyp_bound}, we get that 

\begin{align*}
\Exp_{S \sim \mathrm{Hyp}(p,k,k)} \left[ \left( 1- \frac{S}{k+\sigma^2} \right)^{-n}  \indc{S \ge \tau } \right ] 
& \le  \sum_{s=\floor{\tau} }^{k} \binom{k}{s} \left( \frac{k}{p-k+1} \right)^s e^{- n \log \left( 1 - \frac{s}{k+\sigma^2} \right)  } \\
& \le  \sum_{s=\floor{\tau} }^{k}  \binom{k}{s} e^{ -s \log \frac{p}{k } - n \log \left( 1 - \frac{s}{k+\sigma^2} \right) - s \log \frac{p-k+1}{p} } \\
& =  \sum_{s=\floor{\tau} }^{k}  \binom{k}{s} e^{ g_n(s) - s \log \frac{p-k+1}{p} },
\end{align*}
where $g_n(s)$ is  given by 
$$
g_n(s) \triangleq -s \log \frac{p}{k} - n \log \left( 1 - \frac{s}{k+\sigma^2} \right).
$$

Note that $g_n(s)$ is convex in $s$
for $\tau \le s \le k$. Hence, the maximum of $g_n(s)$ 
over $s \in [\tau, k]$ 
is achieved at either $s=\tau$ or 
$s= k.$  In view of \prettyref{eq:g_function} and Claim~\ref{claim:g_bound}, for all $n \le n^*$. 
$$
g_n(\tau ) \le g_{n^*}( \tau ) = g(\tau ) \le -\frac{1}{2} k \log \frac{p}{k}.
$$
Thus it remains to upper bound $g_n(k)$.

\begin{claim}
Assume $n \leq \frac{1}{2} (1-\alpha) n^*$ for some $\alpha>0$. Then 
$g_n(k) \le -\alpha k \log (p/k)$. 
\end{claim}
\begin{proof}[Proof of the Claim]
For all $n \leq \frac{1}{2} (1-\alpha) n^*$, 
\begin{align*}
g_n(k)&=-k \log \frac{p}{k}-n \log \left(1-\frac{k}{k+\sigma^2} \right)\\
&=-k \log \frac{p}{k}+ \frac{1}{2} (1-\alpha)  n^* \log \left(1+\frac{k}{\sigma^2} \right)\\
&= -k \log \frac{p}{k} +(1-\alpha) k \log \left(\frac{p}{k} \right)\\ 
&= -\alpha k \log \frac{p}{k}.
\end{align*}
 \end{proof}

In view of the above claim and  the assumption that $\alpha \le 1/2$, we conclude that 
for all $n  \leq \frac{1}{2} (1-\alpha) n^*$, 
\begin{align*}
\Exp_{S \sim \mathrm{Hyp}(p,k,k)} \left[ \left( 1- \frac{S}{k+\sigma^2} \right)^{-n}  \indc{S \ge \tau} \right ] &
\leq   \sum_{k = \floor{\tau}}^k \binom{k}{s}  e^{- \alpha  k \log \frac{p}{k}  - s \log \frac{p-k+1}{p} }\\
&\leq e^{- \alpha  k \log \frac{p}{k} } \sum_{s=0}^k \binom{k}{s}  \left( \frac{p}{p-k+1} \right)^s \\
& \leq e^{- \alpha  k \log \frac{p}{k} }  \left( 1+  \frac{p}{p-k+1} \right)^k \\
& \leq e^{- \alpha  k \log \frac{p}{k} + k \log \frac{2-c}{1-c} },
\end{align*} 
where the last equality holds due to the assumption $k \le c p$.  
\end{proof}

\section{Probability of the conditioning event}
In this section, we upper bound the probability that the conditioning event does not happen.

\begin{lemma}\label{lmm:conditioning}
Consider the set $\calE_{\gamma,\tau}$ defined in \eqref{eq:Egamma}. Let $\tau = k (1-\eta)$ for some $\eta \in [0,1]$. 
Then we have
\[
\prob{(X,\beta) \in \calE_{\gamma,\tau}^c}  \le \exp\left\{   - \frac{n\gamma}{4}   +   \eta  k \log\left(  \frac{e^2 p }{ \eta^2 k} \right)  
 \right\} .
\]
Furthermore, for 
$$
\eta = \frac{1}{ \log^2 ( 1+ k/\sigma^2 ) }, \quad \text{ and } \quad 
 \gamma \ge  \frac{ k \log (p/k) }{ n \log ( 1+ k/\sigma^2 ) } \vee \frac{k }{n} 
$$
then there exists a universal constant $C>0$ such that if $k/\sigma^2 \geq C$, then
\[
\prob{(X,\beta) \in \calE_{\gamma,\tau}^c}  \le  \exp\left\{   - \frac{n\gamma}{8} \right\}. 
\]

\end{lemma}
\begin{proof} 
Fix $\beta$ to be a $k$-sparse binary vector in $\{0,1\}^p$.  
Let $\beta'$ denote another $k$-sparse binary vector and $s=\inner{\beta,\beta'}$. We have $X(\beta+\beta') \sim \calN(0, 2(k+s) \identity_{n})$ and therefore  
$$
\frac{ \left\| X(\beta+\beta') \right\|^2 }{ 2(k+s)} \sim \chi^2_n.
$$
Observe also that the number of different $\beta'$ with $\iprod{\beta}{\beta'} \geq \tau$ is at most
$$
\sum_{\ell = 0}^{\floor{\eta k} } \binom{k}{\ell} \binom{p-k}{\ell}
$$ 
by counting on the different choices of positions of the entries where $\beta'$ differ from $\beta$. 
Combining the two observations it follows from the union bound that
\begin{align}
\prob{(X,\beta) \in \calE_{\gamma,\tau}^c \mid \beta } \le Q_{\chi^2_n}\left( n (2+ \gamma) \right) \sum_{\ell = 0}^{\floor{\eta k}} \binom{k}{\ell} \binom{p-k}{\ell} ,
\end{align}
where $Q_{\chi^2_n}(x)$ is the tail function of the chi-square distribution.

For all $x > 0$, we have (see, \eg, \cite[Lemma 1]{laurent2000adaptive}:
\begin{align}
Q_{\chi^2_n}\left( n (1 + \sqrt{x} + x/2) \right) \le e^{ -  \frac{n x}{ 4}  }.
\end{align}
Noting that $ \sqrt{\gamma} + \gamma/2 \le 1+ \gamma$ for all $\gamma>0$,
we see that $Q_{\chi^2_n}( n (2+ \gamma) ) \le \exp\left\{  -  n \gamma /4 \right\}.$

Next, using the inequalities $\binom{a}{b} \le (\frac{ae}{b})^b$ for $a,b \in \mathbb{Z}_{>0}$ with $a<b$, that $x \rightarrow x \log x$  decreases in $(0,\frac{1}{e})$, and 
$\sum_{i=0}^d \binom{m}{i}\le (\frac{me}{d})^d$ for $d,m \in \mathbb{Z}_{>0}$ with $d<m$ (see, \eg, \cite{Kumar10}), we get that
\begin{align*}
\sum_{\ell = 0}^{\floor{\eta k}} \binom{k}{\ell} \binom{p-k}{\ell} & \le  \sum_{\ell = 0}^{\floor{\eta k}}  
\left( \frac{ek}{\ell} \right)^{\ell }  \binom{p-k}{\ell} \\
 & \le \left( \frac{e}{ \eta} \right)^{ \eta k } 
\sum_{\ell = 0}^{\floor{\eta k}} \binom{p-k}{\ell} \\
& \le \left( \frac{e^2 p }{ \eta^2 k } \right)^{ \eta k}.
\end{align*}

Combining the above expressions completes the first part of the proof of the Lemma.

For the second part, note that under our choice of $\eta$,  
$$
 - \frac{n\gamma}{4}   +   \eta  k \log\left(  \frac{e^2 p }{ \eta^2 k }  \right)
 =  - \frac{n\gamma}{4}   +  \frac{  k \left( \log (p/k) + 4 \log \log (1 + k /\sigma^2) + 2  \right) }{  \log^2(1 + k /\sigma^2) } 
$$
Under the choice of $\gamma$, there exists a universal constant $C>0$ such that if 
if $k/\sigma^2 \geq C$, then
\begin{align*}
\frac{n\gamma}{16} & \ge \frac{k \log (p/k) } { \log^2(1 + k /\sigma^2) }  \\
\frac{n\gamma}{16} & \ge \frac{k \left( 4 \log \log (1 + k /\sigma^2) + 2  \right)}{ \log^2(1 + k /\sigma^2) }.
\end{align*}
Combining the last two displayed equation yields that 
$$
- \frac{n\gamma}{4}   +   \eta  k \log\left(  \frac{e^2 p }{ \eta^2 k }  \right) \le - \frac{n \gamma}{8}.
$$ 
This completes the proof of the lemma.

\end{proof}

\iffalse
\begin{lemma}
It holds that 
$$
\inf_{\beta \in \{0,1\}^p: \| \beta\|_0=k}  \prob{F_{\beta}} \ge 1 + o(1).
$$ 
\end{lemma}
\begin{proof}
 
Setting $x= \delta^2 n/4$, we have 
$$
\prob{ \| X(\beta + \beta') \|^2 \ge 2(1 + \delta +  \delta^2/2) n(k+s) }
\le e^{-\delta^2 n/4}.
$$
Moreover,  for a given $\beta$, the number of choices of $\beta'$ 
with overlap $s$  is  $\binom{k}{s}  \binom{p-k}{k-s}$.
In view of the bound
$$
\binom{k}{s}\binom{p-k}{k-s} = \binom{k}{k-s} \binom{p-k}{k-s}
\le (pk)^{k-s}, 
$$
and the union bund, we have
$$
\prob{ \| X(\beta + \beta') \|^2 \ge 2(1 + 2\delta) n(k+s) }
\le e^{-\delta^2 n/4 + (k-s) \log (pk)}.
$$
For $s \ge k-\frac{k}{\log^2 \lambda}$, 
$$
(k-s) \log (pk) \le \frac{k}{\log^2 \lambda} \log (pk) \le  2\frac{k \log p}{\log^2\lambda}
\overset{(a)}{\le}  \frac{1}{8} \delta^2 (1-\epsilon) n^*   \overset{(b)}{\le}  \frac{1}{8}\delta^2 n,
$$
where $(a)$ holds due to $\lambda \to \infty$ and $\log (p/k)=\Omega(\log p)$; $(b)$ follows from
$n = (1-\epsilon) n^*$. Hence, 
$$
\prob{F_{X,\beta}} \ge 1- e^{-\delta^2 n/8} = 1+o(1).
$$
\end{proof}
\fi

\section{The reason why $k = o(p^{1/2})$ is needed for weak detection threshold $n^*$}\label{app:assumption_needed}

This section shows that weak detection between the planted model $P$ and the null model
$Q_\lambda$ is possible for any choice of $\lambda>0$ and for all $n = \Omega_p(n^*)$, if
$k = \Omega_p(p^{1/2})$, $k/\sigma^2=\Omega_p(1)$, and $\log(p/k)=\Omega_p\left(\log(1+k/\sigma^2) \right)$.  In particular, we show the following proposition.

\begin{proposition}
Suppose
\begin{align}
\frac{n k^2}{ p \left( k+\sigma^2 - k^2/p \right) } = \Omega_p(1). \label{eq:weak_detection_upper_bound}
\end{align}
Then weak detection is information-theoretically possible.
\end{proposition}
\begin{remark}
If $k/\sigma^2=\Omega_p(1)$ and $k/p$ is bounded away from $1$, 
then \prettyref{eq:weak_detection_upper_bound} is equivalent to 
\[
\frac{nk}{p} =\Omega_p(1). 
\]
Recall that 
\[
n^* =  \frac{2k \log (p/k)}{\log (1+k/\sigma^2)}. 
\]
Therefore, if furthermore $k = \Omega_p(p^{1/2})$ and  $\log(p/k)=\Omega_p\left(\log(1+k/\sigma^2) \right)$,

then $n^* k/p = \Omega_p(1)$ and hence weak detection is possible for all $n=\Omega_p(n^*)$. 

\end{remark}

\begin{proof}
Let $\bar{\beta} = \expect{ \beta}$ and consider the test statistic 
\[
\calT(X, Y)  = \iprod{Y}{ X\bar{\beta}};
\]
we declare planted model if $\calT(X,Y) \ge 0$ and null model otherwise.  
Let $A, B$ be independent $n$-dimensional standard Gaussian vectors. 
Then we have that 
\begin{align*}
\left( X\bar{\beta}, Y \right)  \overset{d}{=} 
\begin{cases}
\left( \sqrt{k^2/p}\, A, \; \sqrt{k^2/p} \, A + \sqrt{k+\sigma^2 - k^2 /p} \, B \right) & \text{ if } (X, Y) \sim P \\
\left( \sqrt{k^2/p} \, A,  \; \lambda \sigma B \right)  & \text{ if } (X, Y) \sim Q_{\lambda}. 
\end{cases}
\end{align*}
Hence,
$$
Q_{\lambda} \left( \iprod{Y}{ X\bar{\beta}}  \le 0 \right) = \frac{1}{2}, 
$$
and 
\begin{align*}
P\left( \iprod{Y}{ X\bar{\beta}} \le 0 \right) & = \expect{ Q\left( \sqrt{ \frac{k^2/p}{ k+\sigma^2 - k^2/p}  } \|A \| \right)},
\end{align*}
where $Q(x) = \int_x^{\infty} (2\pi)^{-1/2} \exp(-t^2/2) \, \d t$
is the tail function of the standard Gaussian. 

Therefore, as long as $\sqrt{ \frac{k^2/p}{ k+\sigma^2 - k^2/p}  } \|A \| $ does not converge to $0$ in probability, 
then 
$P\left( \iprod{Y}{ X\bar{\beta}} \le 0 \right) \le 1/2-\epsilon$ for some positive constant $\epsilon>0$.
Thus,
$$
P\left( \iprod{Y}{ X\bar{\beta}} < 0 \right) + Q_{\lambda} \left( \iprod{Y}{ X\bar{\beta}} \ge  0 \right) \le 1-\epsilon;
$$
hence weak detection is possible. Since $\|A\|^2_2 \sim \chi^2_n$ highly concentrates on $n$, it follows that if 
\begin{align}
\frac{n k^2}{ p \left( k+\sigma^2 - k^2/p \right) } = \Omega_p(1), \label{eq:weak_detection_upper_bound}
\end{align}
then weak detection is possible.

\end{proof}

\end{appendices}

\bibliographystyle{alpha}

\bibliography{graphical_combined,long_names,support_recovery,bibliographyIZ}

\end{document}